\documentclass[12pt]{article}
\usepackage{amsmath}
\usepackage{amsfonts}
\usepackage{amssymb}
\usepackage{amsthm}
\usepackage{palatino}
\usepackage[margin=2.5cm, vmargin={1.5cm}]{geometry}

\usepackage{pst-node}
\usepackage{pst-coil}
\usepackage{pst-plot}
\usepackage{pst-grad}
\usepackage{pstcol}

\newrgbcolor{LemonChiffon}{1.0 0.98 0.8}
\newrgbcolor{magenta}{1.0 0.3 0.4}
\newrgbcolor{mistyrose}{1.0 0.0 0.6}
\newrgbcolor{deeppink}{1.0 0.08 0.58}
\newrgbcolor{lightsalmon}{1.0 0.63 0.48}
\newrgbcolor{lightred}{0.7 0.0 0.0}
\newrgbcolor{lightblue}{0.68 0.85 0.90}
\newrgbcolor{indianred}{0.80  0.36  0.36}
\newrgbcolor{lightgreen}{0.56 0.93 0.56}
\newrgbcolor{byellow}{0.93 0.76 0.3}

\title{Parabolic, hyperbolic and elliptic Poincar\'e series}

\author{\"Ozlem Imamo$\bar{\text g}$lu, Cormac O'Sullivan}
\date{}

\begin{document}

\maketitle

\def\s#1#2{\langle \,#1 , #2 \,\rangle}

\def\H{{\Bbb H}}
\def\F{{\frak F}}
\def\C{{\Bbb C}}
\def\R{{\Bbb R}}
\def\Z{{\Bbb Z}}
\def\Q{{\Bbb Q}}
\def\N{{\Bbb N}}
\def\st{{\Bbb S}}
\def\D{{\Bbb D}}
\def\fun{{\Bbb F}}
\def\B{{\Bbb B}}
\def\G{{\Gamma}}
\def\GH{{\G \backslash \H}}
\def\g{{\gamma}}
\def\L{{\Lambda}}
\def\ee{{\varepsilon}}
\def\K{{\mathcal K}}
\def\Re{\text{\rm Re}}
\def\Im{\text{\rm Im}}
\def\SL{\text{\rm SL}}
\def\GL{\text{\rm GL}}
\def\PSL{\text{\rm PSL}}
\def\sgn{\text{\rm sgn}}
\def\tr{\text{\rm tr}}

\def\ca{{\frak a}}
\def\cb{{\frak b}}
\def\cc{{\frak c}}
\def\cd{{\frak d}}
\def\ci{{\infty}}

\def\sa{{\sigma_\frak a}}
\def\sb{{\sigma_\frak b}}
\def\sc{{\sigma_\frak c}}
\def\sd{{\sigma_\frak d}}
\def\si{{\sigma_\infty}}

\def\se{{\sigma_\eta}}
\def\sz{{\sigma_{z_0}}}


\newtheorem{theorem}{Theorem}
\newtheorem{lemma}[theorem]{Lemma}
\newtheorem{prop}[theorem]{Proposition}
\newtheorem{cor}[theorem]{Corollary}
\newtheorem{conj}[theorem]{Conjecture}

\newcounter{coundef}
\newtheorem{adef}[coundef]{Definition}

\newcounter{thm1count}
\newtheorem{thm1}[thm1count]{Theorem}

\renewcommand{\labelenumi}{(\roman{enumi})}

\numberwithin{equation}{section}

\bibliographystyle{plain}

\begin{abstract}\noindent
Following Petersson, we study the parabolic, hyperbolic and elliptic expansions of
holomorphic cusp forms and the associated Poincar\'e series. We show how these ideas
extend to the space of second-order cusp forms.
\end{abstract}

\section{Introduction}
 \label{intro1}
Let $\G \subseteq \PSL_2(\R)$ be a Fuchsian
group of the first kind acting on the upper half plane $\H$. We write
$x+iy=z\in \H$ and set $d\mu z$ to be
the $\SL_2(\R)$-invariant hyperbolic volume form $dx dy/y^2$. Assume the volume of the
quotient space
$\G\backslash\H$ is equal to $V < \infty$.  Let $S_{k}(\G)$ be the space of holomorphic
weight $k$
cusp forms for $\G$. This is the vector space of holomorphic
functions $f$ on $\H$ which decay rapidly in each cusp  of
$\G$ and satisfy the transformation property
\begin{equation}\label{trans}
\frac{f(\g z)}{j(\g,z)^{k}} -f(z)=0 \text{ \ \ for all \ \ } \g\in \G,
\end{equation}
with $j(\left(\smallmatrix a
& b \\ c & d \endsmallmatrix\right),z)=cz+d$  for
$ \g=(\smallmatrix a & b \\ c & d \endsmallmatrix)$.  We do not assume that $\G$ has
cusps. If there are none then we may ignore the rapid decay condition for $S_k$. We assume
throughout that the multiplier system is trivial and that, unless otherwise stated, $4
\leqslant k\in 2\Z$.
Using the notation $(f|_k \g)(z):=f(\g z)/j(\g,z)^k$, extended to all $\C[\PSL_2(\R)]$ or
$\C[\SL_2(\R)]$ by linearity, (\ref{trans}) can be written more simply  as $f|_k(\g
-1)=0$.

\vskip 3mm
The  automorphy property (\ref{trans}) may be modified to obtain different families of
objects: for example mixed cusp forms \cite{Lee} or vector-valued cusp forms \cite{KM2}.
In this article we generalize  (\ref{trans}) to define higher-order automorphic forms. These
have arisen independently in different contexts, see for example \cite{FW, Go2, KZ}. They
are also an interesting special case of the vector-valued cusp forms of Knopp and Mason
\cite{KM2, Mas} corresponding to unipotent representations of $\G$. See \cite[Theorem
2.1]{JO'S3} for an explicit example of this connection in the context of higher-order
non-holomorphic Eisenstein series. The first historical appearance of second-order forms appears to be in a paper by Eichler \cite[Section 4]{Ei} where he defines the second-order parabolic Poincar\'e series in connection with  his work on automorphic integrals and their periods, see section \ref{eich}.

\vskip 3mm
Let $\N$ be the natural numbers $\{1,2,3, \dots\}$   and $\N_0 = \N \cup \{0 \}$.

\begin{adef} For $n \in \N_0$, the $\C$-vector space $S^n_k(\G)$ is defined recursively as
follows.
 Let $S_k^0(\G)$ consist only of the  function $\H \to 0$. For $n \geqslant 1$, let
$S_k^n(\G)$ contain all holomorphic functions $f:\H \to \C$  that satisfy
\begin{equation}\label{trans2}
f|_k(\g -1) \in S^{n-1}_k(\G) \text{ \ \ for all \ \
} \g \in \G.
\end{equation}
For all parabolic elements $\pi$ of $\G$ we also require
\begin{equation}\label{trans3}
f|_k(\pi -1) =0.
\end{equation}
Finally $f$ must  decay rapidly in each cusp (this is explained in section
\ref{decayexpl}).
\end{adef}

We call $S^n_k(\G)$ the space of  order $n$, weight $k$ holomorphic cusp forms for $\G$.
 Usually we  just write $S^n_k$. Clearly $S^1_k$ and $S_k$ are synonymous. Induction, as
in \cite[Lemma 3.1]{JO'S3}, shows $S^{n_1}_k \subseteq S^{n_2}_k$ for any two integers $0
\leqslant n_1  \leqslant n_2$. Therefore $S_k \subseteq S^n_k$ and higher-order forms are
a generalization of the usual cusp forms.

\vskip 3mm
The identity in $\G$ is $I=\pm\left(\smallmatrix 1 & 0 \\ 0 & 1
\endsmallmatrix\right)$. The remaining elements may be partitioned into three sets:  the
parabolic, hyperbolic and elliptic elements. These correspond to translations, dilations
and rotations, respectively, in $\H$. As is well-known, the relation (\ref{trans}) for
parabolic elements leads to a Fourier expansion of $f$ associated to each cusp of $\G$.
The parabolic Fourier coefficients that arise often contain a great deal of
number-theoretic information. A family of corresponding parabolic Poincar\'e series can be
constructed whose inner products with $f$  produce these  Fourier coefficients.

\vskip 3mm
Much less well-known are Petersson's  hyperbolic and elliptic Fourier expansions,
introduced in \cite{peter}. In the first half of this paper we start by giving an exposition of Petersson's work. For the benefit of the reader we develop all three expansions
and their associated Poincar\'e series  in  sections \ref{parsec}, \ref{hypsec} and
\ref{ellsec}. We give a unified treatment   with notation that emphasizes the similarities
of the three cases.

\vskip 3mm
The series we construct in sections \ref{parsec}, \ref{hypsec} and \ref{ellsec} are all
examples of
{\it relative Poincar\'e series}.

\begin{theorem} \label{relpoin}
Let $\G_0$ be a subgroup of $\G$ and  $\phi$  a holomorphic function on $\H$ satisfying
$\phi|_k \g = \phi$ for all $\g$ in $\G_0$ and
\begin{equation}\label{check}
\int_{\G_0 \backslash \H} |\phi(z)|y^{k/2} \, d\mu z < \infty.
\end{equation}
Then the relative Poincar\'e series
\begin{equation}\label{relpoincare}
P[\phi](z):= \sum_{\g \in \G_0 \backslash \G} (\phi|_k \g)(z)
\end{equation}
converges absolutely and uniformly on compact subsets of $\H$ to an element of $S_k$.
\end{theorem}
This theorem is stated in \cite{Ka} and proved in \cite[Chapter 1, \S 7]{Krushkal}, for
example. Variations of it also appear explicitly and implicitly in  many other works. The
proof, given in section \ref{se}, simply involves showing that the series
(\ref{relpoincare}) is  bounded by the integral (\ref{check}). 

\vskip 3mm
 In  the second half of the paper we show how these ideas extend naturally to
the second-order space $S^2_k$. In   section \ref{sof} we  prove the analog of Theorem \ref{relpoin}. Let $\operatorname{Hom}_0(\G, \C)$ be the homomorphisms from $\G$ to $\C$ that are $0$ on the parabolic elements of $\G$. The following relative  Poincar\'e series, twisted by such a homomorphism, are second-order forms.   

\begin{theorem} \label{relpoin2}
Let $\G_0$ be a subgroup of $\G$ and  $\phi$  a holomorphic function on $\H$ satisfying
$\phi|_k \g = \phi$ for all $\g$ in $\G_0$. Let $L \in \operatorname{Hom}_0(\G, \C)$ with
$L(\g)=0$ for all $\g$ in $\G_0$. If
\begin{equation*}
\int_{\G_0 \backslash \H} \Big(1+|\Lambda_L^+(z)|+ |\Lambda_L^-(z)|\Big)|\phi(z)|y^{k/2}
\, d\mu z < \infty
\end{equation*}
then
$$
P[\phi,L](z):= \sum_{\g \in \G_0 \backslash \G} L(\g)(\phi|_k \g)(z)
$$
converges absolutely and uniformly on compact subsets of $\H$ to an element of $S^2_k$.
\end{theorem}
The functions $\Lambda_L^+(z)$ and $\Lambda_L^-(z)$ above satisfy
$$
L(\g)=\Lambda_L^+(\g z) - \Lambda_L^+(z) + \overline{\Lambda_L^-(\g z) - \Lambda_L^-(z)}
$$
for all $\g \in \G$ and all $z \in \C$. See (\ref{lamb2}) for their definition.
 
\vskip 3mm
Theorem \ref{relpoin2} allows us to construct parabolic, hyperbolic and elliptic second-order Poincar\'e series.
In section \ref{spn} we show that, whenever they exist,
these Poincar\'e series of order 1 and 2 always span their respective cusp form spaces. In
the final section  we speculate on the situation for third and higher-order forms.

\section{Parabolic expansions} \label{parsec}

\subsection{} \label{decayexpl}
All the properties of Fuchsian groups used in this paper are explained in \cite{Katok},
\cite[Chapter 1]{Sh} and \cite[Chapter 2]{Iw2}.
 We say that $\g$ is a parabolic element of $\G$ if its trace,  $\tr(\g)$, has absolute
value 2. Then $\g$  fixes one point in $\R \cup \{\infty\}$. These parabolic fixed points
are the cusps of $\G$. If $\ca$ is a cusp of $\G$ then the subgroup, $\G_\ca$, of all
elements in $\G$ that fix $\ca$ is isomorphic to $\Z$. Thus $\G_\ca=\langle \g_\ca
\rangle$ for a parabolic generator $\g_\ca \in \G$.  There exists a scaling matrix $\sa
\in \SL_2(\R)$ so that $\sa \infty= \ca$ and
\begin{equation}
  \sa^{-1} \G_\ca \sa =  \left\{ \left. \pm \begin{pmatrix} 1 & m \\ 0 & 1
\end{pmatrix}
\; \right| \; \ m\in {\Z}\right\}. \label{ginf}
\end{equation}
The matrix $\sa$ is unique up to multiplication on the right by any $\left(\smallmatrix 1
& x \\ 0 & 1
\endsmallmatrix\right)$ with $x \in \R$. We label the group (\ref{ginf}) as $\G_\ci$. A
natural fundamental domain for $\G_\ci \backslash \H$ is the set $\fun_\infty$ of all $z \in \H$ with $0
\leqslant \Re(z) < 1$. The image of this set under $\sa$ will be a fundamental domain
for $\G_\ca \backslash \H$ as shown in Figure \ref{pfig}.



\SpecialCoor
\psset{griddots=5,subgriddiv=0,gridlabels=0pt}

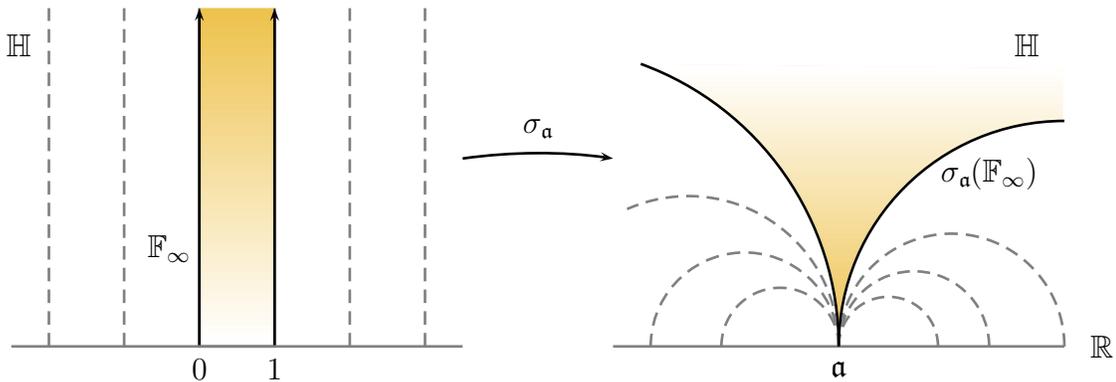
\begin{figure}[h]
\begin{center}
\begin{pspicture}(-0.8,-0.5)(14.8,5) 

\psset{linewidth=1pt}

\pscustom[fillstyle=gradient,linecolor=gray,gradmidpoint=1,gradbegin=byellow,gradend=white,gradlines=50]{%
  \psline[linewidth=1pt](2.5,4.5)(2.5,0)(3.5,0)(3.5,4.5)}

  \psline[linestyle=dashed,linecolor=gray](0.5,0)(0.5,4.5)
  \psline[linestyle=dashed,linecolor=gray](1.5,0)(1.5,4.5)
  \psline[linestyle=dashed,linecolor=gray](4.5,0)(4.5,4.5)
  \psline[linestyle=dashed,linecolor=gray](5.5,0)(5.5,4.5)

  \psline[linewidth=1pt]{<-}(2.5,4.5)(2.5,0)
  \psline[linewidth=1pt]{->}(3.5,0)(3.5,4.5)
  \psline[linewidth=1pt,linecolor=gray](0,0)(6,0)
  \psline[linewidth=1pt,linecolor=gray](8,0)(14,0)

  \psset{shortput=nab}
  \pnode(6,2.5){aa}
  \pnode(8,2.5){bb}
  \ncarc{->}{aa}{bb}^{$\sa$}

  \rput(2.5,-0.3){$0$}
  \rput(3.5,-0.3){$1$}
  \rput(0.1,4.0){$\H$}
  \rput(2.1,1.3){$\fun_\infty$}
  \rput(13.5,4.0){$\H$}
  \rput(13,2.3){$\sa(\fun_\infty)$}
  \rput(14.5,0){$\R$}

  \psarc[linestyle=dashed,linecolor=gray](9,0){2}{0}{114}
  \psarc[linestyle=dashed,linecolor=gray](12.5,0){1.5}{0}{180}
  \psarc[linestyle=dashed,linecolor=gray](12,0){1}{0}{180}
  \psarc[linestyle=dashed,linecolor=gray](9.75,0){1.25}{0}{180}
  \psarc[linestyle=dashed,linecolor=gray](10.22,0){0.78}{0}{180}
  \psarc[linestyle=dashed,linecolor=gray](11.66,0){0.66}{0}{180}

  \pscustom[fillstyle=gradient,linecolor=white,gradmidpoint=1,gradbegin=white,gradend=byellow,gradlines=50,gradangle=0]{%
  \psarcn[linewidth=1pt](7,0){4}{70}{0}
  \psarcn[linewidth=1pt](14,0){3}{180}{90}
  \psline[linewidth=1pt](14,3)(14,3.7)}

  \psarcn[linewidth=1pt](7,0){4}{70}{0}
  \psarcn[linewidth=1pt](14,0){3}{180}{90}


  \rput(11,-0.3){$\ca$}

\end{pspicture}
\caption{The parabolic scaling map \label{pfig}}
\end{center}
\end{figure}


We next define an operator $A$ that converts functions with a particular parabolic
invariance into functions with invariance as $z \to z+1$. Similar, though slightly more
elaborate, operators will do the same for functions with  hyperbolic and elliptic
invariance in sections \ref{hypsec}, \ref{ellsec}.

\begin{lemma} \label{parperiodic}
For any function $f$ with $f|_k \g_\ca = f$, define
$$
A_{\ca}f:=\bigl(f|_k \sa\bigr).
$$
Then $\left(A_{\ca}f \right)(z+1)=\left( A_{\ca}f\right)(z)$.
\end{lemma}
\begin{proof}
We have
\begin{eqnarray*}
\left(A_{\ca}f \right)(z+1) & = & \bigl(f|_k \sa\bigr) ((\smallmatrix 1 & 1 \\ 0 & 1
\endsmallmatrix)z)\\
& = &   \left( \bigl(f|_k \sa(\smallmatrix 1 & 1 \\ 0 & 1 \endsmallmatrix)
\sa^{-1}\bigr)|_k \sa \right) (z)\\
& = & \left( \bigl(f|_k \g_\ca \bigr)|_k \sa \right)  (z)\\
& = & \left(A_{\ca}f \right)(z).
\end{eqnarray*}
\end{proof}

It follows that $f$ in $S_k$ or $S_k^n$ implies $\left(A_{\ca}f \right)(z)$ has period 1
and is holomorphic on $\H$. It consequently has a Fourier expansion
\begin{equation}\label{fexp}
    \left(A_{\ca}f \right)(z)=\sum_{m \in \Z} b_\ca(m) e^{2\pi i m z}.
\end{equation}
The rapid decay condition at the cusp $\ca$ in the definitions of $S_k$ and $S_k^n$ is
then
\begin{equation}\label{expdecay}
    \left(A_{\ca}f \right)(z)= (f|_k \sa)(z) \ll_\ca e^{-c y}
\end{equation}
as $y \to \infty$ uniformly in $x$ for some constant $c>0$. This must hold at each of
the cusps $\ca$. It is equivalent to $f|_k \sa$ only having terms with $m \geqslant 1$ in
the expansion (\ref{fexp}).
We make the following definition, valid for all $f \in S_k^n$.
\begin{adef}
For $f \in S_k^n$, the parabolic expansion of $f$ at $\ca$ is
\begin{equation} \label{parexp}
     \left(f|_k  \sa\right) (z)=\sum_{m \in \N} b_\ca(m) e^{2\pi i m z}.
\end{equation}
\end{adef}

Suppose $\cb = \g \ca$ is another cusp $\G$-equivalent to $\ca$. Then
$$
\G_\cb = \G_{\g \ca} = \g \G_\ca \g^{-1}
$$
and $\sb = \sigma_{\g \ca} =   \g\sa$.

We have
\begin{eqnarray*}
  \left(f|_k  \sb\right) (z) &=& j(\g \sa,z)^{-k} f(\g \sa z) \\
   &=& j(\sa,z)^{-k} (f|_k \g )(\sa z) \\
   &=& j(\sa,z)^{-k}\left(f(\sa z) +f^*(\sa z)\right)\\
   &=& \left(f|_k  \sa\right) (z) +\left(f^*|_k  \sa\right) (z)
\end{eqnarray*}
for $f|_k(\g -1)=f^*$. Thus, if $f \in S_k$ then $f^*=0$ and its parabolic expansions at
$\G$-equivalent cusps are identical. If
 $f \in S_k^n$, its parabolic expansions at $\G$-equivalent cusps are the same up to
addition of the parabolic expansion of an element of $S_k^{n-1}$. Therefore, when testing
whether $f \in S^n_k$, if $f$ satisfies condition (\ref{trans2}) in the definition of
$S_k^n$ then the rapid decay condition  need only be verified  at the finite number of
$\G$-inequivalent cusps of $\G$.

\subsection{}
It is easy to see that the reasoning in Lemma \ref{parperiodic} may be reversed. If $g(z)$
has period $1$ then
\begin{equation}\label{supa}
A_{\ca}^{-1}g :=g|_k (\sa^{-1})
\end{equation}
 satisfies $\left(A_{\ca}^{-1}g \right)|_k \g_\ca = A_{\ca}^{-1}g$. Therefore
$\left(A_{\ca}^{-1}g \right)|_k \g = A_{\ca}^{-1}g$ for all $\g \in \G_\ca$. If we have a
holomorphic function $g$ on $\H$ with period $1$ then setting $\phi = A_{\ca}^{-1}g$ and
$\G_0 = \G_\ca$ in Theorem \ref{relpoin} gives us a natural candidate for a relative
Poincar\'e series. The obvious periodic functions to use are the ones appearing on the
right of (\ref{parexp}): $ e^{2\pi i m z}$.
So, recalling definition (\ref{relpoincare}), we set
\begin{equation}\label{parpoin}
    \Phi_{\text{\rm Par}}(z,m,\ca) := P[A_{\ca}^{-1} e^{2\pi i m \cdot}](z) = \sum_{\g \in
\G_\ca \backslash \G} \left(\left(A_{\ca}^{-1} e^{2\pi i m \cdot}\right)|_k \g \right)
(z)= \sum_{\g \in \G_\ca \backslash \G}
    \frac{e^{2\pi i m (\sa^{-1} \g z)}}
    {j(\sa^{-1} \g , z)^{k}}.
\end{equation}
The next proposition is also proved in \cite[Sections 3.1, 3.2]{Iw2}, for example.

\begin{prop} \label{parinc}
For $4\leqslant k \in 2\Z$ and $m \in \N$ we have $\Phi_{\text{\rm Par}}(z,m,\ca) \in
S_k$.
\end{prop}

\begin{proof}
With Theorem \ref{relpoin} and the discussion leading to (\ref{parpoin}) we need only
confirm that (\ref{check}) holds for $\G_0 = \G_\ca$ and $\phi = A_{\ca}^{-1} e^{2\pi i m
\cdot}$. We have
\begin{eqnarray}
  \int_{\G_\ca \backslash \H} |\phi(z)|y^{k/2} \, d\mu z &=& \int_{\G_\ci \backslash \H}
|\phi(\sa z)|\Im (\sa z)^{k/2} \, d\mu z \nonumber\\
   &=& \int_0^\infty \int_0^1 \left| \frac{e^{2\pi i m z}}{j(\sa^{-1},\sa z)^k} \right|
\frac{y^{k/2}}{|j(\sa, z)|^k} \, d\mu z \nonumber \\
   &=& \int_0^\infty e^{-2\pi my}y^{k/2-2} \, dy \label{para3}
\end{eqnarray}
and (\ref{para3}) is bounded for $m \geqslant 1$ and $k>2$.
\end{proof}
The  vector space $S_k$ is equipped
with a natural inner product, due to Petersson in \cite{Pet39},
\begin{equation}\label{petinn}
\s{f_1}{f_2}:=\int_{\G \backslash \H} f_1(z) \overline{f_2(z)} y^{k} \, d\mu z.
\end{equation}
The following well-known result is demonstrated in \cite[Section 3.3]{Iw2}, for example,
using the unfolding technique. We give the hyperbolic and elliptic versions in the next
sections.

\begin{prop} \label{porth}
For $4\leqslant k \in 2\Z$, $m \in \N$ and $f \in S_k$ satisfying (\ref{parexp})  we have
$$
\s{f}{\Phi_{\text{\rm Par}}(\cdot ,m,\ca)} =  b_{\ca}(m)\left[\frac{(k-2)!}{(4\pi
m)^{k-1}} \right].
$$
\end{prop}

\section{Hyperbolic expansions} \label{hypsec}

\subsection{}
This material is based on Petersson's work in \cite{peter}, see also \cite{Hiramatsu}.
An element $\g$ of $\G$ is hyperbolic if $|\tr(\g)|>2$. Denote the set of all such
elements $\operatorname{Hyp}(\G)$.
Let $\eta=\{\eta_1, \eta_2\}$ be a hyperbolic pair of points in $\R \cup \{\infty\}$ for
$\G$. By this we mean that there exists one element of $\operatorname{Hyp}(\G)$ that fixes
each of $\eta_1$, $\eta_2$. The set of all such $\g$ is a  group which we label $\G_\eta$.
As in the parabolic case this group is isomorphic to $\Z$, see \cite[Theorem 2.3.5]{Katok}
for the proof, and $\G_\eta =\langle \g_\eta \rangle$.
 There exists a  scaling matrix $\se  \in \SL_2(\R)$ such that $\se 0=\eta_1$, $\se
\infty = \eta_2$
and
\begin{equation}\label{etaconj}
\se^{-1} \g_\eta \se = \pm\begin{pmatrix} \xi &  \\ & \xi^{-1} \end{pmatrix}
\end{equation}
for $\xi \in \R$. This scaling matrix $\se$ is unique up to multiplication on the right by any $\left(\smallmatrix x
& 0 \\ 0 & x^{-1}
\endsmallmatrix\right)$ with $x \in \R$.   Replacing the generator $\g_\eta$ by $\g_\eta^{-1}$ if necessary we may
assume $\xi^2 >1$. Let
$$
\fun_{\eta} :=\{z \in \H : 1 \leqslant |z| < \xi^2\}.
$$
Then it is easy to see that $\se( \fun_{\eta} )$ is a fundamental domain for $\G_\eta
\backslash \H$. See Figure \ref{hfig}.
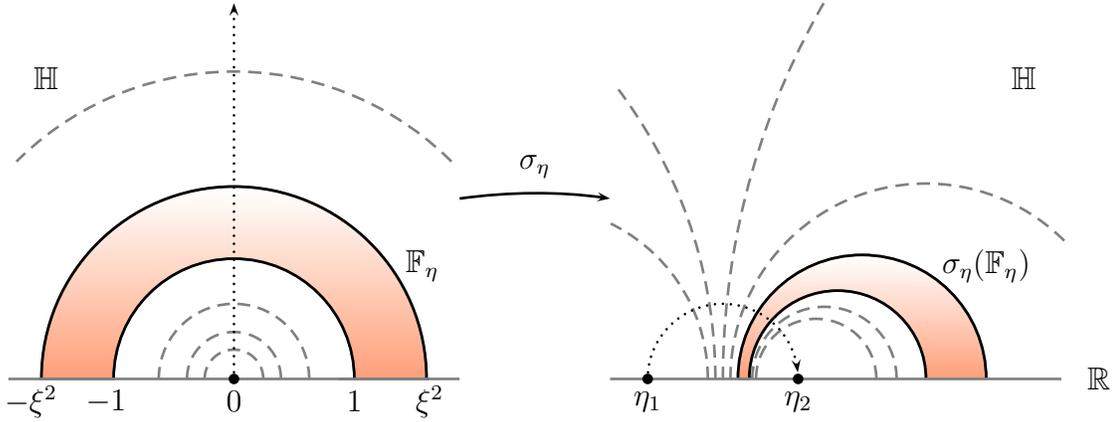
\begin{figure}[h]
\begin{center}
\begin{pspicture}(-0.8,-0.5)(14.8,5) 

\psset{linewidth=1pt}

\pscustom[fillstyle=gradient,linecolor=gray,linewidth=1pt,gradmidpoint=1,gradbegin=lightsalmon,gradend=white,gradlines=50,gradangle=180]{%
  \psline(0.4,0)(1.4,0)
  \psarcn(3,0){1.6}{180}{0}
  \psline(4.4,0)(5.6,0)
  \psarc(3,0){2.56}{0}{180}}

  \psarcn(3,0){1.6}{180}{0}
  \psarc(3,0){2.56}{0}{180}


  \psarc[linecolor=gray,linestyle=dashed](3,0){1}{0}{180}
  \psarc[linecolor=gray,linestyle=dashed](3,0){0.625}{0}{180}
  \psarc[linecolor=gray,linestyle=dashed](3,0){0.39}{0}{180}
  \psarc[linecolor=gray,linestyle=dashed](3,0){4.09}{45}{135}

  \psline[linewidth=1pt,linecolor=gray](0,0)(6,0)

  \psset{shortput=nab}
  \pnode(6,2.4){aa}
  \pnode(8,2.4){bb}
  \ncarc{->}{aa}{bb}^{$\se$}

  \rput(0.5,4){$\H$}
  \rput(13.5,4){$\H$}

  \rput(1.3,-0.3){$-1$}
  \rput(5.6,-0.3){$\xi^2$}
  \rput(0.3,-0.3){$-\xi^2$}
  \rput(4.6,-0.3){$1$}
  \psdot*(3,0)
  \rput(14.5,0){$\R$}
  \rput(5.5,1.5){$\fun_{\eta} $}
  \rput(13.0,1.5){$\se(\fun_{\eta} )$}

  \pscustom[fillstyle=gradient,gradmidpoint=1,gradbegin=white,gradend=lightsalmon,gradlines=50,gradangle=0,linewidth=1pt,linecolor=gray]{%
  \psarc(11.35,0){1.65}{0}{180}
  \psline(9.7,0)(9.85,0)
  \psarcn(11.025,0){1.175}{180}{0}}

  \psarc[linestyle=dashed,linecolor=gray](12.2,0){2.6}{45}{180}
  \psarc[linestyle=dashed,linecolor=gray](10.85,0){0.96}{0}{180}
  \psarc[linestyle=dashed,linecolor=gray](10.74,0){0.8}{0}{180}
  \psarc[linestyle=dashed,linecolor=gray](19.5,0){10}{150}{180}
  \psarc[linestyle=dashed,linecolor=gray](7,0){2.3}{0}{64}
  \psarc[linestyle=dashed,linecolor=gray](3,0){6.4}{0}{37}

  \psarc(11.35,0){1.65}{0}{180}
  \psarcn(11.025,0){1.175}{180}{0}

  \psline[linewidth=1pt,linecolor=gray](8,0)(14,0)

  \psarcn[arcsepB=2pt,linestyle=dotted,dotsep=2pt]{->}(9.5,0){1}{180}{3}
  \psline[linestyle=dotted,dotsep=2pt]{->}(3,0)(3,5)


  \psdot*(8.5,0)
  \psdot*(10.5,0)
  \rput(8.5,-0.3){$\eta_1$}
  \rput(10.5,-0.3){$\eta_2$}
  \rput(3,-0.3){$0$}

\end{pspicture}
\caption{The hyperbolic scaling map \label{hfig}}
\end{center}
\end{figure}


\begin{lemma} \label{hypperiodic}
For any function $f$ with $f|_k \g_\eta = f$, let
$$
\left( A_{\eta}f \right)(z):=\xi^{kz}\bigl(f|_k \se\bigr) (\xi^{2z}).
$$
Then $\left( A_{\eta}f \right)(z+1)=\left( A_{\eta}f \right)(z)$.
\end{lemma}
\begin{proof}
First note that
$$
f\bigl((\smallmatrix \xi &  \\  & \xi^{-1} \endsmallmatrix) z  \bigr) = \xi^{-k}
\bigl(f|_k (\smallmatrix \xi &  \\  & \xi^{-1} \endsmallmatrix)\bigr)(z)
$$
and with a calculation,
\begin{eqnarray*}
\left( A_{\eta}f \right)(z+1) & = & \xi^{k(z+1)}\bigl(f|_k \se\bigr) (\xi^{2(z+1)})\\
& = & \xi^{k}\xi^{kz}\bigl(f|_k \se\bigr) \bigl((\smallmatrix \xi &  \\  & \xi^{-1}
\endsmallmatrix)\xi^{2z}\bigr)\\
& = & \xi^{kz}  \left( \bigl(f|_k \se(\smallmatrix \xi &  \\  & \xi^{-1} \endsmallmatrix)
\se^{-1}\bigr)|_k \se \right)  (\xi^{2z})\\
& = & \xi^{kz}\left( \bigl(f|_k \g_\eta \bigr)|_k \se \right)  (\xi^{2z})\\
& = & \xi^{kz} \bigl(f|_k \se \bigr)  (\xi^{2z})\\
& = & \left( A_{\eta}f \right)(z).
\end{eqnarray*}
\end{proof}

If $f \in S_k$ then $A_{\eta}f $ has period 1 and hence a Fourier expansion:

\begin{equation}\label{fffz}
\left( A_{\eta}f \right)(z)=\sum_{m \in \Z} b_{\eta}(m) e^{2\pi i m z}.
\end{equation}
Put $w=\xi^{2z}$ so that $e^{2\pi i  z}=w^{\pi i/\log \xi}$. Then (\ref{fffz}) implies
that $f \in S_k$ must have the following expansion.
\begin{adef}
The hyperbolic expansion of $f \in S_k$ at $\eta$ is
\begin{equation}\label{hypexp}
\left( f|_k \se \right) (w)= \sum_{m \in \Z} b_{\eta}(m)w^{-k/2+\pi i m/\log \xi}.
\end{equation}
\end{adef}
Petersson introduced this expansion in \cite{peter}. It is also valid for $f \in S_k^n$
provided $f|_k(\g-1) =0$ for all $\g \in \G_\eta$.

\subsection{}

Reversing Lemma \ref{hypperiodic} we see that if $g(z+1)=g(z)$ then the inverse operator
to $A_{\eta}$ acts as follows
\begin{equation}\label{supe}
A_{\eta}^{-1} g:=\left( B_\eta g \right)|_k (\se^{-1})  \text{ \ \ for \ \ } \left( B_\eta
g \right)(z) := z^{-k/2}g\left( \frac{\log(z)}{2\log \xi}\right)
\end{equation}
and $\left( A_{\eta}^{-1} g \right) |_k \g_\eta = A_{\eta}^{-1} g$.
Set
\begin{multline}
    \Phi_{\text{\rm Hyp}}(z,m,\eta)  :=  P[ A_{\eta}^{-1} e^{2\pi i m \cdot}](z) \\
     = \sum_{\g
\in \G_\eta \backslash \G} \left(\left(A_{\eta}^{-1} e^{2\pi i m \cdot}\right)|_k \g
\right) (z) = \sum_{\g \in \G_\eta \backslash \G}
    \frac{(\se^{-1}\g z)^{-k/2+\pi i m /\log \xi}}
    {j(\se^{-1}\g , z)^{k}}. 
\end{multline} \label{hyppoin}

\begin{prop} \label{hypinc}
For $4\leqslant k \in 2\Z$ and $m \in \Z$ we have $\Phi_{\text{\rm Hyp}}(z,m,\eta) \in
S_k$.
\end{prop}
\begin{proof}
With Theorem \ref{relpoin} we must verify (\ref{check}). We have
\begin{eqnarray*}
 \int_{\G_\eta \backslash \H} \left| \frac{(\se^{-1}z)^{-k/2+\pi im/\log
\xi}}{j(\se^{-1},z)^k}\right|
 y^{k/2}\, d\mu z & = & \int_{\fun_{\eta} } \left| w^{-k/2+\pi im/\log \xi}\right|
 \Im(w)^{k/2}\, d\mu w\\
 & = & \int_1^{\xi^2} \int_0^{\pi} r^{-k/2} e^{-\pi m \theta/\log \xi} (r\sin
\theta)^{k/2-2} \, r d\theta dr\\
 & = & 2 \log \xi \int_0^{\pi}  e^{-\pi m \theta/\log \xi} (\sin \theta)^{k/2-2} \,
d\theta.
\end{eqnarray*}

This is bounded for all $m \in \Z$ and $k>2$ as required.
\end{proof}

\begin{prop} \label{horth}
For $4\leqslant k \in 2\Z$, $m \in \Z$  and $f \in S_k$ satisfying (\ref{hypexp}) we have
\begin{eqnarray}
  \s{f}{\Phi_{\text{\rm Hyp}}( \cdot ,m,\eta)} &=& b_{\eta}(m)\left[ 2\log \xi \int_0^\pi
 e^{-2\pi m \theta /\log \xi} (\sin \theta)^{k-2} \, d\theta\right] \label {i1i}\\
   &=&\begin{cases}   b_{\eta}(m) \cdot  \frac{2\log \xi}{(2i)^{k-1}}\left(e^{-2\pi^2 m/\log \xi}-1\right)
\frac{ \G\left(\frac{\pi i m}{\log \xi} -\frac k2 +1\right) \G(k-1)}
   {\G\left(\frac{\pi m}{\log \xi} + \frac k2\right)  } \label {i2i} & \text {if  }m\neq 0\\
  b_{\eta}(0) \cdot  \frac{2 \pi \log \xi}{2^{k-2}}\binom{k-2}{k/2-1} & \text {if } m=0.\\
 \end{cases}
  \end{eqnarray}
\end{prop}

\begin{proof}
Unfold the inner product:
\begin{eqnarray*}
    \s{f}{\Phi_{\text{\rm Hyp}}( \cdot ,m,\eta)} & = & \int_\GH y^k f(z)
\overline{\Phi_{\text{\rm Hyp}}(z,m,\eta)} \, d\mu z\\
& = & \int_\GH y^k f(z) \sum_{\g \in  \G_\eta \backslash \G}
    \overline{\left(\left(A_{\eta}^{-1} e^{2\pi i m \cdot}\right)|_k \g \right)(z) } \,
d\mu z\\
& = & \int_\GH  \sum_{\g \in  \G_\eta \backslash \G} \Im(\g z )^k f(\g z)
    \overline{\left(A_{\eta}^{-1} e^{2\pi i m \cdot}\right) (\g z) } \, d\mu z\\
& = & \int_{\G_\eta \backslash \H}    f(z)
    \overline{\left(A_{\eta}^{-1} e^{2\pi i m \cdot}\right) ( z) } y^k \, d\mu z\\
& = & \int_{\G_\eta \backslash \H}  \left( \frac{\sum_{l \in \Z}
b_{\eta}(l)(\se^{-1}z)^{-k/2+\pi i l /\log \xi}}{j(\se^{-1},z)^{k}} \right)
    \frac{\overline{(\se^{-1} z)^{-k/2+\pi i m /\log \xi}}}
    {\overline{j(\se^{-1},z)}^{k}} y^k \, d\mu z\\
& = & \sum_{l \in \Z} b_{\eta}(l) \int_{\G_\eta \backslash \H}     (\se^{-1}z)^{-k/2+\pi i
l /\log \xi}\overline{(\se^{-1} z)^{-k/2+\pi i m /\log \xi}} \Im(\se^{-1}z)^k \, d\mu z.
\end{eqnarray*}
The integral is
\begin{multline*}
\int_{\fun_{\eta}}     w^{-k/2+\pi i l /\log \xi} \overline{w^{-k/2+\pi i m /\log \xi}}
\Im(w)^k \, d\mu w\\
=\int_1^{\xi^2} \int_0^\pi     (re^{i/\theta})^{-k/2+\pi i l /\log \xi}
\overline{(re^{i/\theta})^{-k/2+\pi i m /\log \xi}} (r\sin \theta)^{k-2} r\, d\theta dr\\
= \ \int_0^\pi   e^{-\theta \pi (l+m) /\log \xi} (\sin \theta)^{k-2} \, d\theta
\int_1^{\xi^2} r^{\pi i (l-m) /\log \xi} \frac{dr}r
\end{multline*}
and letting $u=\log r/\log \xi$ we see that the last integral in $r$ is
$$
\log \xi \int_0^{2} e^{u\pi i (l-m)}\, du = \begin{cases} 2\log \xi & \text{if $l = m$},\\
0 & \text{if $l \neq m$.} \end{cases}
$$
Reassemble to complete the proof of (\ref{i1i}). The integral in (\ref{i1i}) may be evaluated as follows. Let
$$
\mathcal I_{a,b}:=\int_0^\pi e^{a\theta} \sin^b \theta \, d\theta, \text{ \ \ for }a\in \C, b \in 2\N_0.
$$
For all $b \in 2\N_0$ we have
$$
\mathcal I_{0,b}=\int_0^\pi \left( \frac{e^{i\theta}-e^{-i \theta}}{2i} \right)^b \, d\theta
=\frac{\pi}{2^b} \binom{b}{b/2}
$$
using the binomial theorem. The $m=0$ case of 
 (\ref{i2i}) follows. For  $a \neq 0$, we easily have $\mathcal I_{a,0}=(e^{\pi a}-1)/a$.   Using
integration by parts twice we find, for $b \geqslant 2$,
$$
\frac{\mathcal I_{a,b}}{b!} = \frac{1}{a^2+b^2} \frac{\mathcal I_{a,b-2}}{(b-2)!}.
$$
Hence
\begin{eqnarray*}
  \mathcal I_{a,b} &=& \frac{ \mathcal I_{a,0} \cdot b!}{(a^2+b^2)(a^2+(b-2)^2) \cdots (a^2+2^2)} \\
   &=& \frac{ \mathcal I_{a,0} \cdot a \cdot b!}{(2i)^{b+1}\left(\frac{a}{2i}+\frac b2 \right)\left(\frac{a}{2i}+\frac b2 -1\right) \cdots \left(\frac{a}{2i}-\frac b2 \right)} \\
   &=&  \frac{\mathcal I_{a,0} \cdot a \cdot \G(b+1) \G\left(\frac{a}{2i}-\frac b2 \right) }{(2i)^{b+1} \G\left(\frac{a}{2i}+\frac b2 +1 \right)}
\end{eqnarray*}  
and the $m \neq 0$ case of 
 (\ref{i2i}) is proved.
\end{proof}

\subsection{} \label{hyplit}
Various types of hyperbolic series have appeared in the literature.
Associated to a hyperbolic element
$\g_\eta = (\smallmatrix a & b \\ c & d \endsmallmatrix)$ in $\G$ we have the quadratic
form
$$
Q_{\g_\eta}(z):=cz^2+(d-a)z-b
$$
of discriminant $D=(a+d)^2-4>0$ and with zeros $\eta = \{\eta_1, \eta_2\}$ at the fixed
points of $\g_\eta$.  We assume that $\g_\eta$ generates the subgroup of elements fixing
$\eta$, that is, $\G_\eta = \langle \g_\eta \rangle$. Also, since $Q_{\g_\eta}$ depends on the sign
of the matrix entries, in this section we take $\G \subset \SL_2(\R)$ instead of
$\PSL_2(\R)$. Following Katok in \cite{Ka}, we define the  series
\begin{equation}\label{rp}
\theta_{k, \g_\eta}(z) :=
\sum_{\g \in \G_\eta  \backslash \G} \frac{1}{Q_{\g_\eta}(\g z)^{k/2} j(\g,z)^{k}}.
\end{equation}
For $4 \leqslant k \in 2\Z$ it is
 shown in \cite{Ka} that $\theta_{k, \g_\eta} \in S_{k}$ and, moreover, that $S_{k}$ is
spanned by $\theta_{k, \g}$ as $\g$ ranges over all hyperbolic elements of $\G$. In
\cite{GoldmanMillson} they further show that only $\g$ that are words in the group
generators of length at most $2^{k-1}-1$ are required for a spanning set.

\vskip 3mm
Katok uses the series $\theta_{k, \g_\eta}$ to define a hyperbolic rational structure on
$S_{k}$, analogous to the  (parabolic) rational structure associated to the periods of
cusp forms, as in \cite{KoZ}.
For example, with $\G = \SL_2(\Z)$,
$$
C_{k,\g_\eta} \cdot \s{f}{\theta_{k, \g_\eta}}= r_k(f,\g_\eta)
$$
for any $f \in S_{k}$ where the right side is the hyperbolic period of $f$ associated to
$\g_\eta$, defined as
$$
r_k(f,\g_\eta):= \int_{w}^{\g_\eta w} f(z) Q_{\g_\eta}(z)^{k/2-1} \, dz
$$
and independent of $w \in \H$. On the left we have the normalization constant
$$
C_{k,\g_\eta}:=D^{(k-1)/2}\frac{-\sgn( \tr (\g_\eta))}{\pi} \binom{k-2}{k/2-1}^{-1}
2^{k/2-2}.
$$
See \cite{Ka} and  \cite{KoZ} for further details.

\vskip 3mm
 In \cite{Za3}, Zagier encounters the series

\begin{equation}\label{fdisc}
F_{k,D}(z):=\sum_{b^2-4ac =D }  \frac{1}{(az^2+bz+c)^{k/2}}
\end{equation}
for $D \in \N$ where the sum is over all $(a,b,c) \in \Z^3$ with $\gcd(a,b,c)=1$. (We
include the $\gcd$ condition for simplicity.) It is essentially the $D$th Fourier
coefficient of the holomorphic kernel function for the Shimura-Shintani correspondence
between forms of integral and half-integral weight. We can express the quadratic form
appearing in (\ref{fdisc}) as
$$
Q(z) =az^2+bz+c = M_Q \begin{bmatrix}  z \\ 1 \end{bmatrix}:= \begin{pmatrix}  z \\ 1
\end{pmatrix}^t M_Q \begin{pmatrix}  z \\ 1 \end{pmatrix}
$$
for the matrix $M_Q:=\begin{pmatrix} a & b/2 \\ b/2 & c \end{pmatrix}$. We say two
quadratic forms $Q$, $Q'$ are equivalent if $M_{Q'} = \g^t M_Q \g$ for some $\g \in \G =
\SL_2(\Z)$. The equivalence class $[Q]$ is the set of all forms equivalent to $Q$. The
class number $h(D)<\infty$ is the number of equivalence classes of forms of
discriminant $D$. For a quadratic form $Q$ of discriminant $D>0$ put
\begin{equation} \label{fdisc2}
F_{k,D,[Q]}(z)=\sum_{Q' \in [Q]}  \frac{1}{Q'(z)^{k/2}}.
\end{equation}
Thus we may break up $F_{k,D}$ into $h(D)$ pieces $F_{k,D,[Q]}$.

\begin{prop}
For $\G = \SL_2(\Z)$ we have $F_{k,D,[Q_{\g_\eta}]} = \theta_{k,\g_\eta}$.
\end{prop}
\begin{proof}
A short computation shows
$$
Q_{\g_\eta}(z) =   \left(   \begin{pmatrix} 0 & 1 \\ -1 & 0 \end{pmatrix} \g_\eta \right)
\begin{bmatrix}  z \\ 1 \end{bmatrix}.
$$
Therefore
$$
[Q_{\g_\eta}] = \bigcup_{\g \in \G} \left( \g^t \left(\begin{pmatrix} 0 & 1 \\ -1 & 0
\end{pmatrix} \g_\eta \right)\g \right)\begin{bmatrix}  z \\ 1 \end{bmatrix}.
$$
To eliminate the repetition in the union, put
\begin{eqnarray*}
  E &:=& \left\{\g \in \G \ \left| \ \g^t \begin{pmatrix} 0 & 1 \\ -1 & 0 \end{pmatrix}
\g_\eta \g = \begin{pmatrix} 0 & 1 \\ -1 & 0 \end{pmatrix} \g_\eta  \right. \right\} \\
   &=& \left\{\g \in \G \ | \ \g^{-1} \g_\eta \g =  \g_\eta  \right\}.
\end{eqnarray*}
Check that $E$ is a subgroup of $\G$ and use (\ref{etaconj}) to show that $E=\G_\eta$.
Hence
\begin{eqnarray*}
  F_{k,D,[Q_{\g_\eta}]}(z) &=& \sum_{\g \in \G_\eta \backslash \G}
  \left(
    \left(\g^t
        \left(
            \begin{pmatrix} 0 & 1 \\ -1 & 0 \end{pmatrix}
            \g_\eta
        \right)
        \g
    \right)
    \begin{bmatrix}  z \\ 1 \end{bmatrix}
  \right)^{-k/2}
\\
&=& \sum_{\g \in \G_\eta \backslash \G}
    \left(
        \left(
            \begin{pmatrix} 0 & 1 \\ -1 & 0 \end{pmatrix}
            \g_\eta
        \right)
        \left[\g \begin{pmatrix}  z \\ 1 \end{pmatrix}
        \right]
        \right)^{-k/2} \\
   &=& \sum_{\g \in \G_\eta \backslash \G}
   \left(   Q_{\g_\eta}(\g z) j(\g,z)^2
 \right)^{-k/2}.
\end{eqnarray*}
\end{proof}

Now we recognize that the  series (\ref{rp}), (\ref{fdisc}), (\ref{fdisc2}) of Katok and
Zagier  are, up to normalization, the Petersson hyperbolic Poincar\'e series with $m=0$:

\begin{prop}
We have $\theta_{k, \g_\eta}(z) = (\xi-\xi^{-1})^{-k/2}  \Phi_{\text{\rm
Hyp}}(z,0,\eta)$.
\end{prop}
\begin{proof}
We find
\begin{eqnarray*}
Q_{\g_\eta}(z) & = & \left(   \begin{pmatrix} 0 & 1 \\ -1 & 0 \end{pmatrix} \g_\eta
\right) \begin{bmatrix}  z \\ 1 \end{bmatrix}
\\
& = &  \left( \begin{pmatrix} 0 & 1 \\ -1 & 0 \end{pmatrix}  \se\begin{pmatrix} \xi &  \\
& \xi^{-1} \end{pmatrix} \se^{-1} \right) \begin{bmatrix}  z \\ 1 \end{bmatrix}
\\
& = & \left(  (\se^{-1})^t \begin{pmatrix} 0 & -1 \\ 1 & 0 \end{pmatrix}  \begin{pmatrix}
\xi &  \\  & \xi^{-1} \end{pmatrix} \se^{-1} \right) \begin{bmatrix}  z \\ 1
\end{bmatrix}
\\
& = &   \begin{pmatrix} 0 & -\xi^{-1} \\ \xi & 0 \end{pmatrix}   \left[ \se^{-1}
\begin{pmatrix} z \\ 1 \end{pmatrix} \right ]
\\
& = & (\xi-\xi^{-1}) j(\se^{-1} , z)^2 \se^{-1} z .
\end{eqnarray*}
Therefore
$$
Q_{\g_0}(\g z)  = (\xi-\xi^{-1}) \frac {j(\se^{-1} \g, z)^2}{j(\g,z)^2} \se^{-1} \g z
$$
and putting this into (\ref{rp}) and comparing with (\ref{hyppoin}) finishes the proof.
\end{proof}

Finally, in this section, we briefly note that Siegel in \cite[Chapter II, \S 3]{Si} found
that the hyperbolic expansion coefficients of the non-holomorphic parabolic Eisenstein
series are essentially Hecke Gr\"{o}ssencharakter $L$-functions  associated to a real
quadratic field. For second-order non-holomorphic Eisenstein series the same computation
is carried out in \cite{CG}, leading to Hecke $L$-functions twisted by modular symbols.

\section{Elliptic expansions} \label{ellsec}

\subsection{}

If $z_0=\alpha+i \beta$ in $\H$ is fixed by a non-identity element of $\G$ it is called an elliptic
point of $\G$. Such group elements necessarily have traces with absolute value less than 2 (and are called elliptic elements). Let $\G_{z_0} \subset \G$ be the subgroup of all elements fixing $z_0$. As
shown in \cite[Theorem 2.3.5, Corollary 2.4.2]{Katok} it is a cyclic group of finite order
$N>1$.
Let $\varepsilon \in \G$ be a generator of $\G_{z_0}$.  There exists $\sz \in \GL(2,\C)$ so
that $\sz 0 = z_0$, $\sz \infty = \overline{z_0}$.
To be explicit, we take
$$
\sz =  \frac{1}{2i \beta}\begin{pmatrix} -\overline{z_0} & z_0 \\ -1 & 1 \end{pmatrix}, \
\
\sz^{-1} =  \begin{pmatrix} 1 & -z_0 \\ 1 & -\overline{z_0} \end{pmatrix}.
$$
Note that $\sz^{-1} $ maps the upper half plane $\H$  homeomorphically to the open unit
disc $\D_1 \subset \C$ centered at the origin.
For any $w \in \H$ a calculation shows
$
(\sz^{-1}  \varepsilon \sz) w = \zeta^2 w $ with $ \zeta = j(\varepsilon, \overline{z_0})
$
 and
$$
\sz^{-1}  \varepsilon \sz = \begin{pmatrix} \zeta &  \\ & \zeta^{-1} \end{pmatrix}.
$$
Hence  $\zeta$ is a primitive $2N$th root of unity: $\zeta = e^{2\pi i m/(2N)}$ for some
$m$ with $(m,2N)=1$. There exists $m'\in \N$ so that $m'm \equiv 1 \mod 2N$ and
$\zeta^{m'} = e^{\pi i /N}$. So, replacing $\varepsilon$ by $\varepsilon^{m'}$ if
necessary,  we may assume $\zeta = e^{\pi i /N}$. Let $\fun_{z_0} $ equal the
central sector covering $1/N$th of the disc and chosen with angle $\theta$ satisfying
$-\pi/N \leqslant \theta-\pi \leqslant \pi/N$, for example, as in Figure \ref{efig}. Then
$\sz \left(\fun_{z_0}  \right)$ is a convenient fundamental domain for $\G_{z_0} \backslash
\H$. Also note that there exists $C(z_0,\G)>0$ such that
\begin{equation}\label{bndedreg}
    |z| < C(z_0,\G) \text{ \ for all \ } z \in \sz \left(\fun_{z_0}  \right).
\end{equation}
In other words the fundamental domain we have chosen is contained in a bounded region of
$\H$. We will need this in the proof of Theorem \ref{secondconst}.


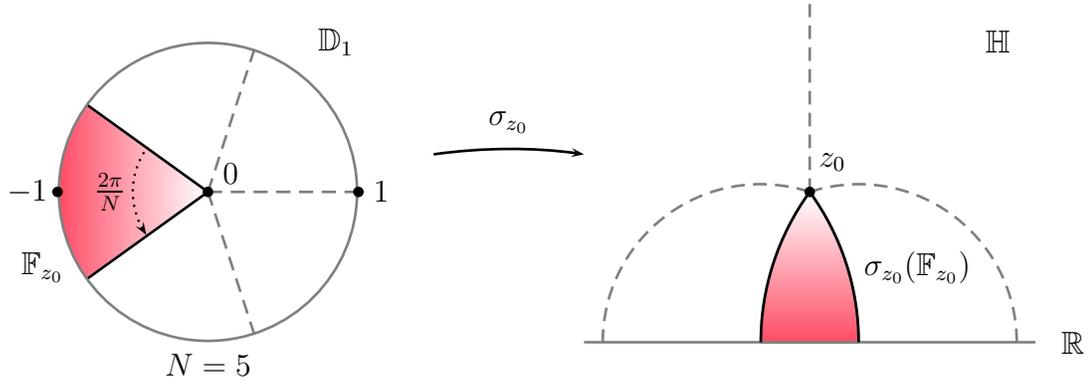
\begin{figure}[ht]
\begin{center}
\begin{pspicture}(-0.8,-0.5)(14.8,4.5) 

\psset{linewidth=1pt}

\pscustom[fillstyle=gradient,linecolor=white,linewidth=1pt,gradmidpoint=1,gradbegin=magenta,gradend=white,gradlines=50,gradangle=90]{%
  \psline(1.383,3.165)(3,2)(1.383,0.836)
  \psarcn(3,2){2}{216}{144}}

  \psline[linewidth=1pt](1.383,3.165)(3,2)(1.383,0.836)

  \psarc[arcsepB=1pt,linestyle=dotted,dotsep=2pt]{->}(3,2){1}{144}{216}
  \psline[linestyle=dashed,linecolor=gray](3,2)(3.618,3.902)
  \psline[linestyle=dashed,linecolor=gray](3,2)(3.618,0.098)
  \psline[linestyle=dashed,linecolor=gray](3,2)(5,2)

  \pscircle[linewidth=1pt,linecolor=gray](3,2){2}
  \psline[linewidth=1pt,linecolor=gray](8,0)(14,0)

  \psset{shortput=nab}
  \pnode(6,2.5){aa}
  \pnode(8,2.5){bb}
  \ncarc{->}{aa}{bb}^{$\sz$}

  \rput(3,-0.3){$N=5$}
  \rput(1.7,2){$\frac{2\pi}{N}$}
  \rput(4.7,4){$\D_1$}
  \rput(13.5,4){$\H$}
  \psdot*(1,2)
  \psdot*(5,2)
  \rput(0.6,2){$-1$}
  \rput(5.3,2){$1$}
  \psdot*(3,2)
  \rput(3.3,2.25){$0$}
  \rput(14.5,0){$\R$}
  \rput(0.8,1){$\fun_{z_0} $}
  \rput(12.43,1){$\sz(\fun_{z_0} )$}

  \psline[linestyle=dashed,linecolor=gray](11,2)(11,4.5)

  \pscustom[fillstyle=gradient,gradmidpoint=1,gradbegin=white,gradend=magenta,gradlines=50,gradangle=0,linewidth=1pt,linecolor=gray]{%
  \psarc(13.752,0){3.402}{144}{180}
  \psline(10.349,0)(11.651,0)
  \psarc(8.248,0){3.402}{0}{36}}

  \psarc[linestyle=dashed,linecolor=gray](10.351,0){2.103}{72}{180}
  \psarc[linestyle=dashed,linecolor=gray](11.649,0){2.103}{0}{108}

  \psarc(13.752,0){3.402}{144}{180}
  \psarc(8.248,0){3.402}{0}{36}
  \psdot*(11,2)

  \rput(11.3,2.4){$z_0$}

\end{pspicture}
\caption{The elliptic scaling map \label{efig}}
\end{center}
\end{figure}


\vskip 3mm
Since any $f \in S_k(\G)$ is holomorphic at $z=z_0$ we see that $f(\sz w)$ is holomorphic
at $w=0$ and has a Taylor series $\sum_n a_{z_0}(n) w^n$. Therefore we get the simple
expansion
\begin{equation}\label{simpleexp}
    f(z)=\sum_{n \in \N_0} a_{z_0}(n) (\sz^{-1} z)^n.
\end{equation}
More useful for our purposes is the slightly different elliptic expansion due to
Petersson. For $f,g:\H \to \C$ define
\begin{eqnarray}
  \left(A_{z_0}f \right)(z) &:=& \zeta^{k z} \bigl(f|_k \sz \bigr) (\zeta^{2 z}),
\nonumber\\
  A_{z_0}^{-1} g &:=& \left(B_{z_0} g \right)|_k (\sz^{-1})  \text{ \ \ for \ \ }
\left(B_{z_0} g \right)(z) := z^{-k/2}g\left( \frac{N\log(z)}{2 \pi i}\right).
\label{supz}
\end{eqnarray}
A calculation identical to that of Lemma \ref{hypperiodic} and its reverse proves the
following.
\begin{lemma} \label{ellperiodic}
We have $A_{z_0}A_{z_0}^{-1}f = A_{z_0}^{-1}A_{z_0}f=f$ and
\begin{eqnarray*}
  (f|_k \varepsilon)(z) = f(z) & \implies & \left(A_{z_0}f \right)(z+1) = \left(A_{z_0}f
\right)(z), \\
  g(z+1)=g(z) & \implies & \left(A_{z_0}^{-1}g \right)|_k \varepsilon  = A_{z_0}^{-1}g.
\end{eqnarray*}
\end{lemma}
Note that the matrices $\sz$ and $\sz^{-1}$ have determinants $1/(2i\beta)$ and $2i\beta$
respectively. In this case it is convenient to normalize the stroke operator $|$ and
define
$$
\left( f|_k \g \right)(z) := \frac{\det(\g)^{k/2} f(\g z)}{j(\g,z)^k}.
$$
Obviously this agrees with our previous definition when $\g \in \SL_2(\R)$.

\vskip 3mm
Let $f \in S_k$ then $A_{z_0}f$ has period 1 and a Fourier expansion
$$
\left( A_{z_0}f \right)(z)=\sum_{m \in \Z} b_{z_0}(m) e^{2\pi i m z}.
$$
Put $w=\zeta^{2 z} = e^{2\pi i z/N}$ so that $e^{2\pi i  z}=w^{N}$ and
$$
\bigl(f|_k \sz\bigr)(w)= \sum_{m \in \Z} b_{z_0}(m)w^{Nm-k/2}.
$$
Since $\bigl( f|_k \sz\bigr)(w)$ is holomorphic at $w=0$ we must have non-negative powers
of $w$ in the above expansion.
Thus any $f \in S_k$ satisfies the following.
\begin{adef}
The elliptic expansion of $f$ in $S_k$ at $z_0$ is
\begin{equation}\label{ellexp}
   \left(f|_k \sz \right)(z)= \sum_{m \in \N \atop Nm-k/2 \geqslant 0}
b_{z_0}(m)z^{Nm-k/2}.
\end{equation}
\end{adef}
We show in Lemma \ref{elliptic} that $f|_k \varepsilon = f$ for all $f$ in $S_k^n$ also.
Thus higher-order cusp forms also have the elliptic expansion (\ref{ellexp}). In some
situations the exponent $Nm-k/2$ is a little awkward and we write
\begin{equation}\label{ellexp2}
   \left(f|_k \sz \right)(z)= \sum_{l \in \N_0} c_{z_0}(l)z^{l}
\end{equation}
instead, where
\begin{equation}\label{ellexp3}
   c_{z_0}(l)= \begin{cases} b_{z_0}\left( \frac{l+k/2}{N} \right) & \text{ if } l \equiv
-k/2 \mod N  \\
0 & \text{ if } l \not\equiv -k/2 \mod N. \end{cases}
\end{equation}

\subsection{}\label{this}

Next define the elliptic Poincar\'e series
\begin{eqnarray}
    \Phi_{\text{\rm Ell}}(z,m,z_0)
      :=  P[ A_{z_0}^{-1} e^{2\pi i m \cdot}](z) & = & \sum_{\g \in \G_{z_0} \backslash \G}
\left(\left( A_{z_0}^{-1} e^{2\pi i m \cdot}\right)|_k \g \right) (z) \label{ellpoin}\\
& = & (2i \beta)^{k/2}\sum_{\g \in  \G_{z_0} \backslash \G}
    \frac{(\sz^{-1} \g z)^{Nm-k/2}}
    {j(\sz^{-1} \g,z)^{k}}. \nonumber
\end{eqnarray}

\begin{prop} \label{ellinc}
For $4\leqslant k \in 2\Z$ and $k/(2N) \leqslant m \in \N$ we have $\Phi_{\text{\rm
Ell}}(z,m,z_0) \in S_k$.
\end{prop}
\begin{proof}
Recalling Theorem \ref{relpoin} we need to verify that
\begin{equation} \label{qwe}
    \int_{\G_{z_0} \backslash \H} |\phi(z)|y^{k/2} \, d\mu z =
(2\beta)^{k/2}\int_{\G_{z_0} \backslash \H} \left|\frac{(\sz^{-1} z)^{Nm-k/2}}{j(\sz^{-1}
,z)^k}\right| y^{k/2} \, d\mu z
\end{equation}
is bounded. Put $w=u+iv=\sz^{-1} z$  and recall that $z_0 = \alpha+i\beta$,  $z=x+iy$. We
calculate
\begin{eqnarray}
    |j(\sz^{-1} ,\sz w)|^{-2} & =& |j(\sz ,w)|^{2} = (2 \beta)^{-2} |1-w|^{2} \nonumber\\
    y & =& \beta \frac{1-|w|^2}{|1-w|^2} \label{yyy}\\
    \left| \frac{\partial (x,y)}{\partial (u,v)}\right| & =& \frac{4 \beta^2}{|1-w|^4}
\nonumber
\end{eqnarray}
and see that (\ref{qwe}) equals
$$
2^{2-k/2} \int_{\fun_{z_0} } |w|^{Nm-k/2} \bigl( 1-|w|^2 \bigr)^{k/2-2} \, du dv
= 2^{2-k/2} \int_0^1 \int_{\pi -\pi/N}^{\pi +\pi/N}  r^{Nm-k/2+1} \bigl( 1-r^2
\bigr)^{k/2-2} \, d\theta dr.
$$
This is bounded for $k>2$ and, by Theorem \ref{relpoin}, $\Phi_{\text{\rm
Ell}}(z,m,z_0) \in S_k$.
\end{proof}

As in the parabolic and hyperbolic cases, the elliptic Poincar\'e series can be used to
determine the elliptic expansion coefficients. We follow closely the reasoning in
\cite{peter} and first prove the following.
\begin{lemma} \label{lemab}
For any integers $a,b \geqslant 0$ and $k \geqslant 2$,
\begin{equation}\label{hypint}
\int_\H \frac{(\sz^{-1} z)^a (\overline{\sz^{-1} z})^b}{|j(\sz^{-1} ,z)|^{2k} }\, y^k \,
d\mu z =
\begin{cases} \displaystyle \frac{4\pi (k-2)! a!}{(4\beta)^k (a+k-1)!} & \text{if $a =
b$}\\ 0 & \text{if $a \neq b$.} \end{cases}
\end{equation}
\end{lemma}
\begin{proof}
 With $w=\sz^{-1} z$, the integral in (\ref{hypint}) becomes
\begin{equation*}
    \int_{\D_1} \frac{(w)^a (\overline{w})^b}{|j(\sz^{-1} ,\sz w)|^{2k} }\, y^{k-2} \left|
\frac{\partial (x,y)}{\partial (u,v)}\right|  \, du dv.
\end{equation*}
As in the proof of Proposition \ref{ellinc} this reduces to
\begin{equation}\label{hypint3}
\frac{4}{(4\beta)^k}\int_{\D_1} (w)^a (\overline{w})^b  \bigl(1-|w|^2\bigr)^{k-2}  \, du
dv.
\end{equation}
Let $w=r e^{i\theta}$ and (\ref{hypint3}) becomes
$$
\frac{4}{(4\beta)^k}\int_0^1 \int_0^{2\pi}  r^{a+b+1} (1-r^2)^{k-2} e^{i\theta(a-b)} \,
d\theta dr.
$$
Of course
$$
\int_0^{2\pi}   e^{i\theta(a-b)} \, d\theta = \begin{cases} 2\pi & \text{if $a = b$}\\ 0 &
\text{if $a \neq b$,}  \end{cases}
$$
and, on repeated integration by parts,
$$
\int_0^1  r^{2a+1} (1-r^2)^{k-2}  \,   dr = \frac{(k-2)! a!}{2 (a+k-1)!}.
$$
Reassemble these pieces to complete the proof.
\end{proof}

A similar proof (integrating over $\fun_{z_0}$ instead of $\D_1$) shows

\begin{lemma} \label{lemab2}
For any integers $Nl-k/2, Nm-k/2 \geqslant 0$ and $k \geqslant 2$,
\begin{equation}\label{hypint2}
\int_{\G_{z_0} \backslash \H} \frac{(\sz^{-1} z)^{Nl-k/2} (\overline{\sz^{-1}
z})^{Nm-k/2}}{|j(\sz^{-1} ,z)|^{2k} }\, y^k \, d\mu z =
\begin{cases} \displaystyle \frac{4\pi (k-2)! (Nm-k/2)!}{N(4\beta)^k (Nm+k/2-1)!} &
\text{if $l = m$}\\ 0 & \text{if $l \neq m$}. \end{cases}
\end{equation}
\end{lemma}

\begin{prop} \label{eorth}
For $f \in S_k$ satisfying (\ref{ellexp}) and $Nm-k/2 \geqslant 0$ we have
$$
\s{f}{\Phi_{\text{\rm Ell}}(\cdot ,m,z_0)} =  b_{z_0}(m) \left[\frac{ \pi (k-2)!
(Nm-k/2)!}{ 2^{k-2} N (Nm+k/2-1)!}\right].
$$
\end{prop}
\begin{proof}
Unfold the inner product as in Proposition \ref{horth}:
\begin{eqnarray*}
    \s{f}{\Phi_{\text{\rm Ell}}( \cdot ,m,z_0)} & = & \int_{\G \backslash \H} y^k f(z)
\overline{\Phi_{\text{\rm Ell}}(z,m,z_0)} \, d\mu z\\
& = & \int_{\G_{z_0} \backslash \H}    f(z)
    \overline{ \left(A_{z_0}^{-1} e^{2\pi i m \cdot} \right)(z)} y^k \, d\mu z\\
& = & (2 \beta)^{k}\int_{\G_{z_0} \backslash \H}   \left( \frac{\sum_{l \in \Z}
b_{z_0}(l)(\sz^{-1} z)^{Nl-k/2}}{j(\sz^{-1} ,z)^{k}} \right)
    \frac{(\overline{\sz^{-1}  z})^{Nm-k/2}}
    {\overline{j(\sz^{-1} ,z)}^{k}} y^k \, d\mu z\\
& = &  (2 \beta)^{k} \sum_{l \in \Z} b_{z_0}(l) \int_{\G_{z_0} \backslash \H}
\frac{(\sz^{-1} z)^{Nl-k/2}(\overline{\sz^{-1}  z})^{Nm-k/2}}{|j(\sz^{-1} ,z)|^{2k}}
    y^k \, d\mu z\\
& = &  b_{z_0}(m) \left[\frac{ \pi (k-2)! (Nm-k/2)!}{ 2^{k-2} N (Nm+k/2-1)!}\right].
\end{eqnarray*}
\end{proof}

\subsection{}
Petersson in \cite{peter} defined the elliptic Poincar\'e series slightly differently as, essentially,
$$
\Phi^*_{\text{\rm Ell}}(z,l,z_0)
       :=(2 i \beta)^{k/2} \sum_{\g \in   \G}
    \frac{(\sz^{-1} \g z)^{l}}
    {j(\sz^{-1} \g,z)^{k}}
$$
for $l \in \N_0$. Notice that the sum is now over all elements of $\G$. With the results
of section \ref{this} and recalling (\ref{ellexp2}) we see that
$$
\s{f}{\Phi^*_{\text{\rm Ell}}(\cdot ,l,z_0)} =  c_{z_0}(l) \left[\frac{\pi (k-2)!
m!}{2^{k-2} (m+k-1)!}\right].
$$
Thus (\ref{ellexp3}) implies that $\Phi^*_{\text{\rm Ell}}(\cdot ,l,z_0)$ is orthogonal to
all of $S_k$, and hence zero, unless $l \equiv -k/2 \mod N$. For such $l$
\begin{eqnarray*}
  \Phi^*_{\text{\rm Ell}}(z ,l,z_0) &=&  (2 i \beta)^{k/2} \sum_{\g \in  \G \backslash
\G_{z_0}} \sum_{\varepsilon \in   \G_{z_0}}
    \frac{(\sz^{-1} \g \varepsilon z)^{l}}
    {j(\sz^{-1} \g \varepsilon,z)^{k}} \\
   &=& \sum_{i=0}^{N-1} \left( \Phi_{\text{\rm Ell}}(\cdot ,l,z_0)|_k \varepsilon^i
\right) (z) \\
   &=& N \cdot \Phi_{\text{\rm Ell}}(z,l,z_0).
\end{eqnarray*}
For $l \in \N_0$ we have shown
$$
\Phi^*_{\text{\rm Ell}}(z,l,z_0)= \begin{cases} N \cdot \Phi_{\text{\rm
Ell}}(z,(l+k/2)/N,z_0) & \text{ if } l \equiv -k/2 \mod N  \\
0 & \text{ if } l \not\equiv -k/2 \mod N. \end{cases}
$$

\section{Relative Poincar\'e Series} \label{se}

We now give the proof of the relative Poincar\'e series construction, Theorem
\ref{relpoin}.

\begin{thm1}
Let $\G_0$ be a subgroup of $\G$ and  $\phi$  a holomorphic function on $\H$ satisfying
$\phi|_k \g = \phi$ for all $\g$ in $\G_0$ and
\begin{equation*}
\int_{\G_0 \backslash \H} |\phi(z)|y^{k/2} \, d\mu z < \infty.
\end{equation*}
Then the relative Poincar\'e series
$$
P[\phi ](z):= \sum_{\g \in \G_0 \backslash \G} (\phi|_k \g)(z)
$$
converges absolutely and uniformly on compact subsets of $\H$ to an element of $S_k$.
\end{thm1}
\begin{proof}
Define the open hyperbolic ball of center $z_0$ and radius $r$ as
$$
\B(z_0,r) := \{z\in \H : \rho(z,z_0)<r\}
$$
with hyperbolic distance $\rho(z,z_0)=\log\bigl( (1+|\sz^{-1} z|)/(1-|\sz^{-1} z|)\bigr)$.
Then
$$
\sz^{-1} \B(z_0,r) = \D_R
$$
for
$$
R=\frac{e^r-1}{e^r+1}, \ \ \ \D_R=\{z \in \C : |z|<R \}.
$$
Now a proof very similar to that of Lemma \ref{lemab} shows
\begin{equation}\label{hi1}
\int_{\B(z_0,r)} (\sz^{-1} z)^n \frac{y^{k/2}}{|z-\overline{z_0}|^k }\,  d\mu z =
\begin{cases} C_{r, k} \beta^{-k/2} & \text{if $n=0$}\\ 0 & \text{if $n \in \N$,}
\end{cases}
\end{equation}
for  $k\geqslant 2$, $\beta=\Im(z_0)$ and
$$
C_{r, k} = \frac{4\pi}{2^{k/2}(k-1)}\left( 1-(1-R^2)^{k-1} \right).
$$
Recall the expansion (\ref{simpleexp})
\begin{equation*}
    f(z)=\sum_{n \in \N_0} a_{z_0}(n) (\sz^{-1} z)^n.
\end{equation*}
and note that $a_{z_0}(0) =f(z_0)$. With (\ref{hi1}) we find
$$
a_{z_0}(0)C_{r, k} \beta^{-k/2} = \int_{\B(z_0,r)} f(z)
\frac{y^{k/2}}{|z-\overline{z_0}|^k }\,  d\mu z
$$
giving a type of hyperbolic holomorphic mean-value result (replace $z$ by $w$ and $z_0$ by
$z$):
$$
y^{-k/2}f(z)= \frac 1{C_{r, k}} \int_{\B(z,r)} f(w) \frac{\Im(w)^{k/2}}{|w-\overline{z}|^k
}\,  d\mu w
$$
valid for any $f$ holomorphic on $\B(z,r)$ with real $k \geqslant 2$.
Since $\Im(w)>0$ and $\Im(\overline{z})=-y$ we have $|w-\overline{z}|>y$ and
therefore
\begin{equation}\label{referen}
y^{k/2}|f(z)| \leqslant \frac 1{C_{r, k}} \int_{\B(z,r)} |f(w)| \Im(w)^{k/2}\,  d\mu w.
\end{equation}
We have $\g \B(z_0,r) = \B(\g z_0,r)$ for all $\g$ in $\G$.
As shown in \cite[Proposition 1.8]{Sh}, there exists an $r>0$ so that
$$
\B(\g z_0,r) \cap \B(\g' z_0,r) = \emptyset
$$
for all $\g \neq \g' \in \G$. With this choice of $r$ we see that
\begin{eqnarray*}
   y^{k/2}|P[\phi ](z)| = \left|y^{k/2} \sum_{\g \in \G_0 \backslash \G} (\phi|_k \g)(z)
\right| & \leqslant &
    \sum_{\g \in \G_0 \backslash \G} \Im(\g z)^{k/2} |\phi( \g z)|\\
    & \leqslant &
    \sum_{\g \in \G_0 \backslash \G} \frac 1{C_{r, k}} \int_{\B(\g z,r)} |\phi(w)|
\Im(w)^{k/2}\,  d\mu w\\
& \leqslant &
     \frac 1{C_{r, k}} \int_{\G_0 \backslash \H} |\phi(w)| \Im(w)^{k/2}\,  d\mu w
<\infty
\end{eqnarray*}
so $P[\phi ]$  converges absolutely and uniformly to a holomorphic function on $\H$.
At this point we may verify that $P[\phi ]|_k(\g-1)=0$ for all $\g \in \G$.
Also we have shown that $y^{k/2}|P[\phi ](z)|$ is bounded and it follows from this, see
\cite[p. 70]{Iw2} for example, that $P[\phi ]$ has rapid decay at each cusp. With all this
 $P[\phi ] \in S_k$ as we wanted to show.
\end{proof}

\section{An elliptic expansion example}
For $q=e^{2\pi i z}$, the discriminant function is
$$
\Delta(z) := q \prod_{n=1}^\infty (1-q^{n})^{24}.
$$
It generates the one-dimensional space $S_{12}(\G_0(1))$ for $\G_0(1) = \PSL_2(\Z)$. Its
parabolic expansion  at infinity is
\begin{equation}\label{tau}
\Delta(z) =\sum_{m=1}^\infty \tau(m) q^m
\end{equation}
and Ramanujan discovered that its  coefficients $\tau(m) \in \Z$ have many remarkable
properties.
The point $z_0 = i \in \H$ is an elliptic point for $\G_0(1)$. It is fixed by $\varepsilon
= \pm (\smallmatrix 0 & -1 \\ 1 & 0 \endsmallmatrix)$ with order $N=2$. Then, recalling
(\ref{ellexp2}), the elliptic expansion of $\Delta$ at $i$ is
\begin{equation}\label{elltau}
\bigl( \Delta|_{12} \sigma_i \bigr) (w) =  \sum_{m = 0}^\infty c_{i}(m)w^{m}.
\end{equation}
Do the coefficients $c_{i}(m)$ have any arithmetic properties? As Petersson realized in
\cite{peter}, it is possible to prove a general result relating the elliptic expansion
coefficients at a point to the Taylor coefficients there.

\begin{prop} \label{bim} For $f \in S_k$ with elliptic expansion (\ref{ellexp2}) at $z_0
\in \H$ we have
$$
c_{z_0}(m)= \sum_{r=0}^m \binom{m+k-1}{r+k-1} \frac{(z_0 - \overline{z_0})^{r+k/2}}{r!}
f^{(r)}(z_0).
$$
\end{prop}
\begin{proof}
Since $f(z)$ is holomorphic in a neighborhood of $z_0$ it has a Taylor expansion
\begin{equation}\label{tay}
f(z)=f(z_0)+(z-z_0)f'(z_0)+(z-z_0)^2f''(z_0)/2!+ \cdots.
\end{equation}
With (\ref{ellexp2}) we have
\begin{equation}\label{abc}
\bigl( f|_k \sz \bigr) (w) = \frac{(2i\beta)^{k/2}}{(1-w)^{k}} f\left(
\frac{z_0-\overline{z_0} w}{1-w}\right)= \sum_{m = 0}^\infty c_{z_0}(m)w^{m}.
\end{equation}
Putting (\ref{tay}) and (\ref{abc}) together produces
\begin{equation}\label{abc2}
\sum_{m = 0}^\infty c_{z_0}(m)w^{m}= \sum_{j=0}^\infty
\frac{f^{(j)}(z_0)}{j!}(2i\beta)^{j+k/2} \frac{w^j}{(1-w)^{j+k}}.
\end{equation}
Use the well-known identity
\begin{equation}\label{comb}
    \frac 1{( 1-w )^{a}} = \sum_{l=0}^\infty \binom{a-1+l}{a-1} w^l
\end{equation}
for $a \in \N$ in (\ref{abc2}) and compare the coefficients of $w^m$ to complete the
proof.
\end{proof}

Applying the proposition to  $\Delta(z)$ at $z_0=i$ we find
\begin{eqnarray}
c_i(m) & = & \sum_{r=0}^m \binom{11+m}{11+r}  (2i)^{r+6} \Delta^{(r)}(i)/r!
\nonumber\\
& = & -2^{6} \sum_{r=0}^m \binom{11+m}{11+r} \sum_{n=1}^\infty \tau(n) e^{-2\pi n} (-4\pi
n)^r/r! \label{del}
\end{eqnarray}
It is clear from (\ref{del}) that $c_i(m) \in \R$ and from (\ref{ellexp3}) we know that
$c_i(m)=0$ for $m$ odd.
Evaluating the $c_i(m)$ numerically we have
$$
\bigl( \Delta|_{12} \sigma_i \bigr) (w) \approx  -0.114  +1.094 w^2 - 2.621 w^4
- 6.694 w^6 +37.787 w^8   + O(w^{10}).
$$
With the  Chowla-Selberg formula \cite[p. 110]{SeCh} we may recognize the first term as
$$
c_i(0) = -(4\pi)^{-6} \left( \frac{\G(1/4)}{\G(3/4)}\right)^{12}.
$$
We will return to this interesting topic in future work.

\section{The dimension of $S^n_k$} \label{sndord}

The first-order space $S_k$ is finite dimensional. See \cite[Theorem 2.24]{Sh} for an
exact formula for $\dim S_k$ in terms of the signature of $\G$.
We begin our study of higher-order forms, in this section, by showing that  $S^n_k$ being
finite dimensional follows from $\G$ being finitely generated.
  Suppose $\Gamma \backslash \H$ has genus $g,$ $r$ elliptic fixed
points
and $p$ cusps, then there are $2g$ hyperbolic elements $\gamma_i,$
$r$ elliptic elements $\varepsilon_i$ and $p$ parabolic elements
$\pi_i$ generating $\Gamma$ and satisfying the $r+1$
relations:
\begin{equation}\label{gens}
    [\g_1, \g_{g+1}]\dots
[\g_{g}, \g_{2g}]\varepsilon_1 \dots \varepsilon_r \pi_1 \dots
\pi_{p}=1, \qquad \varepsilon_j^{e_j}=1
\end{equation}
for $1 \leqslant j \leqslant r$ and integers $e_j \geqslant 2$ as in \cite[Proposition
2.6]{Iw}.
Here $[a, b]$ denotes the
commutator $aba^{-1}b^{-1}$ of $a, b$.

\begin{lemma} \label{elliptic}
For every elliptic element $\varepsilon$ in $\G$ and every $f$ in $S^n_k$ we have
$f|_k(\varepsilon -1)=0$.
\end{lemma}
\begin{proof}
We use an induction argument. The lemma is true in the case $n=1$ by definition. Now
suppose $f \in S^n_k$ for $n>1$. Let $g=f|_k (\varepsilon-1)$. Then $g \in S^{n-1}_k$
and, by induction, $g|_k (\varepsilon-1)=0$. Also
\begin{eqnarray*}
    g|_k(\varepsilon^i-1) & = & g|_k (\varepsilon-1)(\varepsilon^{i-1}+ \cdots +
\varepsilon^1 + 1) \\
    & = & \bigl(g|_k (\varepsilon-1)\bigr) |_k (\varepsilon^{i-1}+ \cdots + \varepsilon^1
+ 1) \\
    & = & 0
\end{eqnarray*}
 for every $i \geqslant 1$. If $\varepsilon$ has order $N$ then
\begin{eqnarray*}
    g & = & f|_k(\varepsilon -1) \\
    & = & f|_k(\varepsilon^{N+1} -1) \\
    & = & f|_k(\varepsilon -1)(\varepsilon^{N}+ \cdots + \varepsilon^2+ \varepsilon^1 + 1)
\\
    & = & g|_k(\varepsilon^{N}+ \cdots + \varepsilon^2+ \varepsilon^1 + 1) \\
    & = & \sum_{i=1}^N g|_k(\varepsilon^i -1) + (N+1)g \\
    & = &  (N+1)g.
\end{eqnarray*}
Hence $g=0$ and $f|_k(\varepsilon -1)=0$ as required.
\end{proof}

We noted in the introduction that $S^{n_1}_k \subseteq S^{n_2}_k$ for any two integers $0
\leqslant n_1  \leqslant n_2$.
So we may  consider the map
\begin{equation*}
\mathcal P_n: S^{n}_k /S^{n-1}_k \to \left( S^{n-1}_k /S^{n-2}_k \right)^{2g}
\end{equation*}
given by
\begin{equation}\label{pnmap2}
f \mapsto (f|_k (\g_1-1), f|_k (\g_2-1), \cdots ,f|_k (\g_{2g}-1))
\end{equation}
with the hyperbolic generators of (\ref{gens}).
\begin{lemma}
The map $\mathcal P_n$ is  well defined, linear and one-to-one.
\end{lemma}
\begin{proof}
To see that $\mathcal P_n$ is well defined we note that if $f, g \in S^{n}_k$ represent
the same element in $S^{n}_k /S^{n-1}_k$ then $f-g=h$ for $h \in S^{n-1}_k$. Hence $f|_k
(\g_i-1) - g|_k (\g_i-1) = h|_k (\g_i-1) \in S^{n-2}_k$ and each component of $\mathcal
P_n$ is well defined.
The map $\mathcal P_n$ is clearly linear. To show it is one-to-one we examine $\ker
(\mathcal P_n)$. If $f \in \ker (\mathcal P_n)$ then $f|_k (\g_i-1)  \in S^{n-2}_k$ for
every hyperbolic generator $\g_i$. By definition $f|_k (\pi_i-1)  =0$ for all parabolic
generators and, with Lemma \ref{elliptic}, $f|_k (\varepsilon_i-1)  =0$ for all elliptic
generators. Therefore
$f|_k (\g-1)  \in S^{n-2}_k$ for all $\g \in \G$ and $f  \in S^{n-1}_k$. Thus $\ker
(\mathcal P_n)=0$ in $S^{n}_k /S^{n-1}_k$ and the map is one-to-one.
\end{proof}

We see then that
$$
\dim(S^{n}_k /S^{n-1}_k) \leqslant 2g \dim(S^{n-1}_k /S^{n-2}_k)
$$
and
\begin{equation}\label{findim}
\dim(S^{n}_k /S^{n-1}_k) \leqslant (2g)^{n-1} \dim(S^{1}_k ).
\end{equation}
It  follows that $S^{n}_k /S^{n-1}_k$ is finite dimensional for all orders $n$. We may
also write (\ref{findim}) as
$$
\dim(S^{n}_k ) \leqslant \dim(S^{n-1}_k) + (2g)^{n-1} \dim(S^{1}_k )
$$
so that
$$
\dim(S^{n}_k ) \leqslant \frac{(2g)^n-1}{2g-1} \dim(S^{1}_k )
$$
 and $S^{n}_k$ is also finite dimensional.
Similar arguments appear in \cite[p. 452]{KZ} and \cite[Theorem 2.3]{CDO}.

\vskip 3mm
 In the case where $\G$ has a parabolic element
the following more precise result is demonstrated in \cite{DO'S2}:
\begin{theorem} \label{dim2}
 For $k$ in $2\Z$ and $\G\backslash\H$ non compact with genus $g$ we have
\begin{eqnarray*}
\dim S^2_{k} & = & 0 \text{ \ if \ } k \leqslant 0,\\
\dim S^2_2 & = & \begin{cases} 0 \text{ \ if \ }\dim S_2=0, \\ (2g+1)\dim S_2 -1\text{ \
otherwise,} \end{cases}\\
\dim S^2_{k} & = & (2g+1)\dim S_{k} \text{ \ if \ }k \geqslant 4.
\end{eqnarray*}
\end{theorem}
Diamantis and Sim in \cite{DS} have recently  extended Theorem \ref{dim2} to all higher
orders.
For example, the following formula may be derived from their \cite[Theorem 4.1]{DS}.
Suppose $n \in \N$ and $4 \leqslant k \in 2\Z$. If $g \geqslant 2$ is the genus of non
compact $\GH$ we have
$$
\dim S^n_{k} =  \left \lfloor \frac{G^{n+1}}{2(G-g)(G-1)}\right \rfloor  \cdot \dim S_{k}
$$
when $G=g+\sqrt{g^2-1}$. They also handle the difficult weight $2$ case. Their method of
proof involves constructing non-holomorphic parabolic Poincar\'e series that have the
desired  order $n$ transformation properties and that also depend on a parameter $s \in
\C$. The elements of $S^n_k$ are obtained by meromorphically continuing these series to
$s=0$. The technical details quickly become formidable.

\vskip 3mm
In the next section we show how Petersson's ideas extend smoothly into order $2$. As
described in the final section, we expect these results to generalize to all orders and
help further our understanding of $S_k^n$.

\section{Constructing second-order forms} \label{sof}
\subsection{}
To construct higher-order Poincar\'e series we extend the constructions (\ref{parpoin}),
(\ref{hyppoin}) and (\ref{ellpoin}) by  including a map $L:\G \to \C$ as follows. Recall
the definitions (\ref{supa}), (\ref{supe}),  (\ref{supz}) and set
\begin{eqnarray}
\label{parpoin2}
    \Phi_{\text{\rm Par}}(z,m,\ca;L) & := & \sum_{\g \in \G_\ca \backslash \G}
    L(\g) \left(\left( A_\ca^{-1} e^{2\pi i m \cdot}\right)|_k \g \right) (z),\\
\label{hyppoin2}
    \Phi_{\text{Hyp}}(z,m,\eta;L)
& := & \sum_{\g \in \G_\eta  \backslash \G}
    L(\g) \left(\left( A_\eta^{-1} e^{2\pi i m \cdot}\right)|_k \g \right) (z),\\
\label{ellpoin2}
\Phi_{\text{Ell}}(z,m,z_0;L)
    & := & \sum_{\g \in  \G_{z_0} \backslash \G}
    L(\g) \left(\left( A_{z_0}^{-1} e^{2\pi i m \cdot}\right)|_k \g \right) (z).
\end{eqnarray}
If $L$ is the constant map  $L:\G \to 1$, then we recover the first-order series. Let
$\operatorname{Hom}(\G, \C)$ be the space of homomorphisms from $\G$ to $\C$ and let
$\operatorname{Hom}_0(\G, \C)$ be the subspace of maps that are zero  on all parabolic
elements of $\G$. As we see in this section, for $L \in \operatorname{Hom}_0(\G, \C)$ the
series (\ref{parpoin2}), (\ref{hyppoin2}), (\ref{ellpoin2}) are second-order cusp forms.
(For the hyperbolic series (\ref{hyppoin2}) to be well-defined we also require $L:\G_\eta
\to 0$.)

\vskip 3mm
We need to prove an analog of Theorem \ref{relpoin} on relative Poincar\'e series.
First, for $L \in \operatorname{Hom}_0(\G, \C)$ define $\Lambda_L^+(z)$ and
$\Lambda_L^-(z)$ as follows. By a well-known theorem of Eichler and Shimura there exist
unique $f^+$, $f^-$ in $S_2(\G)$ so that
$$
L(\g) = \int_z^{\g z} f^+(w) \, dw + \overline{\int_z^{\g z} f^-(w) \, dw}.
$$
The right-side above is independent of $z$ and the path of integration in $\H$. Set

\begin{equation}\label{lamb2}
\Lambda_L^+(z):=\int_i^{z} f^+(w) \, dw, \quad \Lambda_L^-(z):=\int_i^{z} f^-(w) \, dw.
\end{equation}
Then clearly, for all $z$ in $\H$,
$$
L(\g)= \Lambda_L^+(\g z) - \Lambda_L^+(z) + \overline{\Lambda_L^-(\g z)} -
\overline{\Lambda_L^-( z)}.
$$

\begin{thm1} 
Let $\G_0$ be a subgroup of $\G$ and  $\phi$  a holomorphic function on $\H$ satisfying
$\phi|_k \g = \phi$ for all $\g$ in $\G_0$. Let $L \in \operatorname{Hom}_0(\G, \C)$ with
$L(\g)=0$ for all $\g$ in $\G_0$. If
\begin{equation}\label{check2}
\int_{\G_0 \backslash \H} \Big(1+|\Lambda_L^+(z)|+ |\Lambda_L^-(z)|\Big)|\phi(z)|y^{k/2}
\, d\mu z < \infty
\end{equation}
then
$$
P[\phi,L](z):= \sum_{\g \in \G_0 \backslash \G} L(\g)(\phi|_k \g)(z)
$$
converges absolutely and uniformly on compact subsets of $\H$ to an element of $S^2_k$.
\end{thm1}
\begin{proof}
 Set
 \begin{eqnarray*}
   P^+[\phi,L](z) &:=& \sum_{\g \in \G_0 \backslash \G} \Lambda_L^+(\g z)(\phi|_k \g)(z),
\\
   P^-[\phi,L](z) &:=& \sum_{\g \in \G_0 \backslash \G} \overline{\Lambda_L^-(\g
z)}(\phi|_k \g)(z)
 \end{eqnarray*}
so that
$$
P[\phi,L](z) = P^+[\phi,L](z) + P^-[\phi,L](z) - \Lambda_L^+(z) P[\phi](z) -
\overline{\Lambda_L^-(z)}P[\phi](z).
$$
Then, using (\ref{check2}) and Theorem \ref{relpoin} with $\phi(z)$ replaced by
$\Lambda_L^+(z)\phi(z)$, we have $P^+[\phi,L] \in S_k$. The series $P^-[\phi,L]$ is not
holomorphic but it does satisfy $P^-[\phi,L]|_k(\g -1)=0$ for all $\g \in \G$ and, using
the proof of Theorem \ref{relpoin} and (\ref{check2}), we obtain
\begin{equation*}
y^{k/2}|P^-[\phi,L](z)| \leqslant \sum_{\g \in \G_0 \backslash \G} \Im(\g z)^{k/2}
|\Lambda_L^-(\g z)\phi(\g z)| < \infty.
\end{equation*}
Now $f \in S_k$ implies $y^{k/2}|f(z)| \ll 1$ uniformly for all $z \in \H$.
Therefore
\begin{eqnarray}
    y^{k/2}|P[\phi,L](z)| & \leqslant & y^{k/2}|P^+[\phi,L](z)| + y^{k/2}|P^-[\phi,L](z)|
\\
    & & \quad + y^{k/2}|\Lambda_L^+(z) P[\phi](z)| + y^{k/2}|\Lambda_L^-(z) P[\phi](z)|
\nonumber\\
& \ll & 1 + |\Lambda_L^+(z)| + |\Lambda_L^-(z)|. \label{cross}
\end{eqnarray}
Let $\F$ be a fixed fundamental domain for $\GH$ intersecting $\R \cup \infty$ at a finite
number of cusps. For each such cusp $\ca$ the scaling matrix $\sa$ maps $z \in \H$ with
$0\leqslant x <1$ and $y$ large into a neighborhood of $\ca$, as in Figure \ref{pfig}.
Since $f^+$, $f^-$ are bounded on $\F$ and have exponential decay at cusps we must have
$|\Lambda_L^+(z)|$, $|\Lambda_L^-(z)| \ll 1$ on $\F$ so that (\ref{cross}) implies
\begin{equation}\label{bndd}
y^{k/2}|P[\phi,L](z)| \ll 1 \text{ \ \ for \ \ }z\in \F.
\end{equation}
It follows that $P[\phi,L](z)$ converges absolutely, and uniformly on all compact sets of
$\H$ to a holomorphic function. It is easy to check that
$$
P[\phi,L]|_k(\g -1) = -L(\g ) P[\phi] \in S_k^1(\G)
$$
for all $\g \in \G$. Since $L \in \operatorname{Hom}_0(\G,\C)$ we also have $
P[\phi,L]|_k(\pi -1) =0$ for all parabolic $\pi \in \G$.

\vskip 3mm
It only remains to check that $P[\phi,L]$ has rapid decay at each cusp.
At any cusp $\ca$ of $\F$ we have a Fourier expansion
\begin{equation}\label{fedex}
    j(\sa,z)^{-k} P[\phi,L](\sa z) = \sum_{m\in \Z} b_\ca(m) e^{2\pi i m z}
\end{equation}
but
$$
|j(\sa,z)^{-k} P[\phi,L](\sa z)| \leqslant y^{-k/2}|\Im(\sa z)^{k/2} P[\phi,L](\sa z)| \ll
y^{-k/2}
$$
for $z$ with $0 \leqslant x<1$ and $y$ large. Consequently $j(\sa,z)^{-k} P[\phi,L](\sa
z) \to 0$ as $y \to \infty$ uniformly for $0 \leqslant x < 1$. Hence we must have
$b_\ca(m)=0$ in (\ref{fedex}) for all $m \leqslant 0$ and $P[\phi,L]$ has rapid decay at
the cusp $\ca$. By the discussion at the end of section \ref{decayexpl} it is enough to
verify the rapid decay condition at the cusps of $\F$.
\end{proof}

\noindent
{\it Remark.} We note that the bound (\ref{bndd}) is not true in general if we allow $z
\in \H$. For example, setting $U(z)=y^{k/2}|P[\phi,L](z)|$ we see that
$$
U(\g^m z) = y^{k/2}|P[\phi,L](z) - m L(\g)P[\phi](z)|
$$
becomes unbounded as $m \to \infty$ if $L(\g)$ and $P[\phi](z)$ are non-zero.

\subsection{} \label{eich}
Now we have everything in place to verify our constructions.

\begin{theorem} \label{secondconst}
For $4 \leqslant k \in 2\Z$ and $L \in \operatorname{Hom}_0(\G, \C)$ we have
\begin{eqnarray}
\Phi_{\text{\rm Par}}(z,m,\ca;L) & \in & S^2_k \text{ \ \ for \ \ }m \in \N \label{aaa1}\\
\Phi_{\text{\rm Hyp}}(z,m,\eta;L) & \in & S^2_k \text{ \ \ for \ \ }m \in \Z, \
L(\g_\eta)=0 \label{aaa2}\\
\Phi_{\text{\rm Ell}}(z,m,z_0;L) & \in & S^2_k \text{ \ \ for \ \ }m \in \N, \ Nm-k/2
\geqslant 0 \label{aaa3}
\end{eqnarray}
where $N$ in (\ref{aaa3}) is $|\G_{z_0}|$, the order of the subgroup of elements fixing $z_0$.
\end{theorem}

The second-order parabolic series $\Phi_{\text{\rm Par}}(z,m,\ca;L)$ appears first, for $m=0$ and $\G=\PSL_2(\Z)$, in section 4 of Eichler's paper 
\cite{Ei}. In fact, he allows $L$ to be any period polynomial (the degree $0$ polynomials correspond to $\operatorname{Hom}_0(\G, \C)$) and uses the second-order parabolic series to prove the existence of general abelian integrals with prescribed periods.  The general result (\ref{aaa1}) is shown, with a different proof to the one given here,  in \cite[Prop. 4.2]{DO'S2}.

\begin{proof}
{\bf The parabolic series.} Beginning with  $\Phi_{\text{\rm Par}}(z,m,\ca;L)$, we proceed
as in Proposition \ref{parinc}. According to Theorem \ref{relpoin2} we need to verify the
condition
\begin{equation}\label{verp}
    \int_0^\infty \int_0^1  \Big(1+|\Lambda_L^+(\sa z)|+ |\Lambda_L^-(\sa z)|\Big)
e^{-2\pi my}y^{k/2-2} \, dx dy < \infty.
\end{equation}
We require a lemma to show the bound
\begin{equation}\label{lambnd}
\Lambda_L^+(\sa z) \ll 1+|\log y|,
\end{equation}
when $|\Re(z)| \leqslant 1$, say, with an implied constant depending only on $L$, $\sa$
and $\G$.

\begin{lemma} \label{logy}
If $z=x+iy$ and $z_0=x_0+i y_0$ are in  $\H$ with $|x-x_0| \leqslant C$ and $f\in S_2(\G)$
then
$$
\int_{z_0}^z f(w)\, dw \ll 1 + |\log y|
$$
with an implied constant independent of $z$ (and depending on $z_0$, $C$, $f$ and $\G$).
\end{lemma}
\begin{proof}
We know that $y|f(z)|$ is bounded for $z\in \H$ \cite[p. 70]{Iw2} so
\begin{eqnarray*}
  \left| \int_{z_0}^z f(w)\, dw  \right| & \ll & \left| \int_{z_0}^{x+i y_0} \Im(w)^{-1}
\, dw + \int_{x+i y_0}^{x+iy} \Im(w)^{-1} \, dw \right| \\
   & \leqslant & \left| \int_{x_0}^{x} y_0^{-1} \, du \right|+ \left|  \int_{y_0}^{y}
\frac{ dv}{v} \right| \\
   & \leqslant & C/y_0 + |\log y_0| + |\log y|.
\end{eqnarray*}
\end{proof}

Now we have
\begin{equation}\label{l123}
\Lambda_L^+(\sa z) = \int_{\sa^{-1} i}^z f^+(\sa w) \, d \sa w = \int_{\sa^{-1} i}^z g(w)
\, dw
\end{equation}
for $g=f^+|_{\sa} \in S_2(\sa^{-1} \G \sa)$. Apply Lemma \ref{logy} to the right side of
(\ref{l123})  to get (\ref{lambnd}). The same bound applies to $\Lambda_L^+(\sa z)$ and
the left side of (\ref{verp}) becomes
$$
\int_0^\infty   \Big(1+|\log y |\Big) e^{-2\pi my}y^{k/2-2} \, dy.
$$
We leave it to the reader to show that this is bounded for $k>2$ and $m \in \N$.

\vskip 3mm
{\bf The hyperbolic series.} For the  series $\Phi_{\text{\rm Hyp}}(z,m,\eta;L)$ we need
to check that
\begin{equation}\label{verh}
\int_1^\mu \int_0^{\pi} \Big(1+\left|\Lambda_L^+\left(\se
\left(re^{i\theta}\right)\right)\right|+ \left|\Lambda_L^-\left(\se
\left(re^{i\theta}\right)\right)\right|  \Big) r^{-1} e^{-\pi m \theta/\log \xi} (\sin
\theta)^{k/2-2} \,  d\theta dr < \infty
\end{equation}
is satisfied. The bounds
$$
\Lambda_L^+(\se z), \Lambda_L^-(\se z) \ll 1+|\log y|,
$$
with an implied constant depending only on $L$, $\sa$ and $\G$ are proved in the same way
as (\ref{lambnd}). With these the reader may confirm that (\ref{verh}) is true for
$k>2$ and $m \in \Z$.

\vskip 3mm
{\bf The elliptic series.} For the  series $\Phi_{\text{\rm Ell}}(z,m,z_0;L)$ we need
\begin{equation}\label{vere}
\int_0^1 \int_{\pi -\pi/N}^{\pi +\pi/N}
\Big(1+\left|\Lambda_L^+\left(\sz \left(re^{i\theta}\right)\right)\right|+
\left|\Lambda_L^-\left(\sz \left(re^{i\theta}\right)\right)\right|  \Big)
 r^{m+1} \bigl( 1-r^2 \bigr)^{k/2-2} \, d\theta dr < \infty.
\end{equation}
As we saw in (\ref{bndedreg}) and Figure \ref{efig},  $\sz \left(re^{i\theta}\right)$ will
be contained in a bounded region of $\H$. Thus we may apply Lemma \ref{logy} to get
$$
\Lambda_L^+\left(\sz \left(re^{i\theta}\right)\right) \ll 1+\left|\log \Im \left(\sz
\left(re^{i\theta}\right)\right)\right|.
$$
Similarly for $\Lambda_L^-$ and, using (\ref{yyy}), the left side of (\ref{vere}) is
bounded by a finite constant times
\begin{equation}\label{hje}
\int_0^1 \int_{\pi -\pi/N}^{\pi +\pi/N}
\left(1+\left|\log \Big( \frac{1-r^2}{|1-re^{i\theta}|^2} \Big)\right|  \right)
 r^{m+1} \bigl( 1-r^2 \bigr)^{k/2-2} \, d\theta dr .
\end{equation}
We have
$$
\left|\log \Big( \frac{1-r^2}{|1-re^{i\theta}|^2} \Big)\right| \leqslant
\left|\log (1-r^2)\right|
+ 2\left|\log |1-re^{i\theta}| \right|
$$
and since $N \geqslant 2$ we know
$$
\left|\log |1-re^{i\theta}| \right| \leqslant \log 2.
$$
Consequently, (\ref{hje}) is bounded by a constant times
$$
\int_0^1
\left(1+\left|\log (1-r^2)\right|  \right)
 r^{m+1} \bigl( 1-r^2 \bigr)^{k/2-2} \,  dr .
$$
It is straightforward to verify that this is bounded for $k>2$ and $m \in \N_0$.
\end{proof}

\section{Spanning Questions} \label{spn}

\subsection{} \label{spansp1}
If $\G$ contains a parabolic element, or equivalently if $\G\backslash\H$ is non-compact
with a cusp $\ca$, then let
$$
\mathcal S^1_{{\text{\rm Par}}}(\ca) := \big\{ \Phi_{\text{\rm Par}}(z,m,\ca) \ \big| \ m
\in \N \big\}
$$
be the set of all weight $k$ parabolic Poincar\'e series associated to this cusp.
The elements in this set span the weight $k$ cusp forms
\begin{equation}\label{span1}
    S_k=\langle \mathcal S^1_{{\text{\rm Par}}}(\ca) \rangle
\end{equation}
 for the simple reason that any element in a  subspace orthogonal to all of  $\mathcal
S^1_{{\text{\rm Par}}}(\ca)$ must have an identically zero parabolic expansion at $\ca$ by
Proposition \ref{porth}.

\vskip 3mm
If $\G$ does not contain parabolic elements then we cannot construct parabolic Poincar\'e
series.  If $\G$ has an elliptic  fixed point $z_0$ then similar reasoning with
Proposition \ref{eorth} shows that the weight $k$ elliptic Poincar\'e series
$$
\mathcal S^1_{{\text{\rm Ell}}}(z_0) := \big\{ \Phi_{\text{\rm Ell}}(z,m,z_0) \ \big| \ m
\in \N_0 \big\}
$$
span $S_k$.

\vskip 3mm
It is possible that $\G$ does not contain any parabolic or elliptic elements. It must
contain a hyperbolic element though, by \cite[Theorem 2.4.4]{Katok}. If $\eta=\{\eta_1,
\eta_2\}$ is any pair of hyperbolic fixed points then, with   Proposition \ref{horth},
$S_k$ is spanned by the weight $k$ hyperbolic Poincar\'e series
$$
\mathcal S^1_{{\text{\rm Hyp}}}(\eta) := \big\{ \Phi_{\text{\rm Hyp}}(z,m,\eta) \ \big| \ m
\in \Z \big\}.
$$

\subsection{}

We would expect the weight $k$ second-order Poincar\'e series to span the space
$S_k^2/S_k^1$ and this is indeed the case, provided we allow the homomorphisms $L$ to vary
as well as the indices $m$. Similarly to the first-order sets in section \ref{spansp1},
define
\begin{eqnarray*}
  \mathcal S^2_{{\text{\rm Par}}}(\ca) &:=& \big\{ \Phi_{\text{\rm Par}}(z,m,\ca;L) \ \big|
\ m \in \N, \ L \in \operatorname{Hom}_0(\G,\C) \big\} \\
  \mathcal S^2_{{\text{\rm Ell}}}(z_0) &:=& \big\{ \Phi_{\text{\rm Ell}}(z,m,z_0;L) \ \big|
\ m \in \N_0, \ L \in \operatorname{Hom}_0(\G,\C) \big\}.
\end{eqnarray*}
Recall the hyperbolic generators (\ref{gens}) for $\G$. For $1 \leqslant i \leqslant 2g$
define the homomorphisms $L_{\g_i} \in$
 Hom$_0(\G, \C)$ dual to these generators by
$$
L_{\g_i}(\g_j) := \delta_{ij}
$$
and each $L_{\g_i}$ is zero on the other elliptic and parabolic generators. It follows
that $\{ L_{\g_i} : 1 \leqslant i \leqslant 2g\}$ is a basis for  Hom$_0(\G, \C)$.

\begin{prop} \label{ord2gen}
For a cusp $\ca$ of $\G\backslash\H$ or an elliptic fixed point $z_0$,  the sets $\mathcal
S^2_{{\text{\rm Par}}}(\ca)$, $\mathcal S^2_{{\text{\rm Ell}}}(z_0)$ each span
$S_k^2/S_k^1$.
\end{prop}
\begin{proof}
Recall the map $\mathcal P_2:S_k^2/S_k^1 \to (S_k^1)^{2g}$ given by (\ref{pnmap2}). A
short calculation shows that
$$
\frac{\Phi_{\text{\rm Par}}(\g_j z,m,\ca;L_{\g_i})}{j(\g_j,z)^k} = \Phi_{\text{\rm
Par}}(z,m,\ca;L_{\g_i}) - L_{\g_i}(\g_j)\Phi_{\text{\rm Par}}(z,m,\ca).
$$
Hence
$$
\mathcal P_2\bigl( \Phi_{\text{\rm Par}}(z,m,\ca;L_{\g_i}) \bigr)=(0,0, \dots
,\underbrace{-\Phi_{\text{\rm Par}}(z,m,\ca)}_{\text{$i$th component}}, \dots,0).
$$
Let $F$ be any element of $S_k^2$ and
$$
\mathcal P_2(F)=(f_1, f_2, \dots , f_{2g}).
$$
Since we know that the $\Phi_{\text{\rm Par}}(z,m,\ca)$ span $S_k$ we may write each $f_i$
as a finite sum
$$
f_i(z)=\sum_m c_i(m) \Phi_{\text{\rm Par}}(z,m,\ca).
$$
Therefore
$$
\mathcal P_2\left(F(z)+\sum_{i=1}^{2g} \sum_m c_i(m) \Phi_{\text{\rm 
Par}}(z,m,\ca;L_{\g_i})\right) =0
$$
from which it follows that
$$
F(z)+\sum_{i=1}^{2g} \sum_m c_i(m) \Phi_{\text{\rm Par}}(z,m,\ca;L_{\g_i}) \in S_k^1
$$
and $S_k^2/S_k^1$ is spanned by $\mathcal S^2_{{\text{\rm Par}}}(\ca)$ as we wanted to
show. The proof for $\mathcal S^2_{{\text{\rm Ell}}}(z_0)$ is identical.
\end{proof}

The hyperbolic version of this result is slightly more involved in that, as well as $L$
and $m$, we also need to vary the hyperbolic points $\eta$. For a simple choice, put
\begin{eqnarray*}
  \mathcal A &:=& \big\{ \Phi_{\text{\rm Hyp}}(z,m,\eta(\g_1);L_{\g_i}) \ \big| \ m \in \Z,
\ 2\leqslant i \leqslant 2g \big\} \\
  \mathcal B &:=& \big\{ \Phi_{\text{\rm Hyp}}(z,m,\eta(\g_2);L_{\g_1}) \ \big| \ m \in \Z
\big\}.
\end{eqnarray*}

\begin{prop}
The set $\mathcal S^2_{{\text{\rm Hyp}}} = \mathcal A \cup \mathcal B$  spans
$S_k^2/S_k^1$.
\end{prop}
\begin{proof}
With similar reasoning to the proof of Proposition \ref{ord2gen} we see that $\mathcal A$
spans the subspace of $S_k^2/S_k^1$ with image  $(0,f_2,f_3,\dots, f_{2g})$ under
$\mathcal P_2$ and $\mathcal B$ spans the subspace  with image  $(f_1,0,\dots, 0)$, where
the $f_i$ are arbitrary elements of $S_k^1$. Therefore, as in Proposition \ref{ord2gen},
for any $F \in S_k^2$ there exists $G$ in the space spanned by $\mathcal A \cup \mathcal
B$ with $\mathcal P_2 (F+G)=0$. Consequently $F+G \in S_k^1$ and $S_k^2/S_k^1$ is spanned
by $\mathcal A \cup \mathcal B$.
\end{proof}

\vskip 3mm
\noindent
{\it Remark.}  We saw in section \ref{hyplit} that the set
$$
\mathcal T^1_{{\text{\rm Hyp}}} := \big\{ \Phi_{\text{\rm Hyp}}(z,0,\eta(\g)) \ \big| \ \g
\in \operatorname{Hyp}(\G) \big\}
$$
also spans $S_k$.
The obvious second-order analog is
$$
\mathcal T^2_{{\text{\rm Hyp}}} := \big\{ \Phi_{\text{\rm Hyp}}(z,0,\eta(\g);L) \ \big| \
\g \in \operatorname{Hyp}(\G),\ L \in \operatorname{Hom}_{0}(\G,\C), L(\g)=0 \big\}.
$$
We expect that this set spans $S_k^2/S_k^1$ but have been unable to prove it.

\vskip 3mm
The proof that the first-order Poincar\'e series span $S_k$ relies on Petersson's inner
product (\ref{petinn}) and that each type of these series - parabolic, hyperbolic and
elliptic - pick out corresponding expansion coefficients, as we saw in Propositions
\ref{porth}, \ref{horth} and \ref{eorth}. Is there a similar proof that second-order
Poincar\'e series span $S^2_k$? The answer is yes, but the inner product on $S^2_k$ is not
given by Petersson's formula (\ref{petinn}). See \cite{IO'S2} for the details.

\section{Higher-order forms} \label{higher}
We saw in section \ref{sof} that the usual Poincar\'e series in $S^1_k$ could be
generalized to series in $S^2_k$ by including $L \in  \operatorname{Hom}(\G,\C)$ in their
definition. Together these series spanned $S^2_k$. What plays the role of $L$ when we look
to construct third and higher-order forms?

\vskip 3mm
First, we generalize our $|$ notation slightly. For any $L:\G \to \C$ and $\g \in \G$
define $L|\g:=L(\g)$ and extend this linearly to $\C [\SL_2(\R)]$ and $\C [\PSL_2(\R)]$.
Consider the functions $L_1, \ L_2:\G \to \C$ satisfying
\begin{eqnarray*}
  L_1|(\delta_1 -1) &=& 0 \text{ \ \ for all \ \ }\delta_1 \in \G \\
  L_2|(\delta_1 -1)(\delta_2 -1) &=& 0 \text{ \ \ for all \ \ }\delta_1, \delta_2 \in \G
\end{eqnarray*}
and in general, for $n \in \N$, let $\operatorname{Hom}^{[n]}(\G,\C)$ be the space of all
functions $L_n:\G \to \C$ where
$$
L_n \left|\bigl((\delta_1 -1) \cdots (\delta_n -1)\bigr) \right.= 0 \text{ \ \ for all \ \
}\delta_1, \dots, \delta_n \in \G.
$$
Clearly $\operatorname{Hom}^{[1]}(\G,\C)$ is just the space of constant functions, and is
isomorphic to $\C$. The functions $L_2 \in \operatorname{Hom}^{[2]}(\G,\C)$ satisfy
$$
L_2(\delta_1 \delta_2) = L_2(\delta_1) + L_2(\delta_2) - L_2(I)\text{ \ \ for all \ \
}\delta_1, \delta_2 \in \G
$$
and we have $L_2 -L_2(I) \in \operatorname{Hom}(\G,\C)$. Hence
$$
\operatorname{Hom}^{[2]}(\G,\C) \cong \C \oplus \operatorname{Hom}(\G,\C).
$$
Similarly, if we set
$$
\operatorname{Hom}^{[n]}_0(\G,\C):=\left\{ \left. L_n \in \operatorname{Hom}^{[n]}(\G,\C) \
\right|  \ L_n|(\pi -1)= 0 \text{ \ for all parabolic \ }\pi \in \G \right\}
$$
then
$$
\operatorname{Hom}^{[2]}_0(\G,\C) \cong \C \oplus \operatorname{Hom}_0(\G,\C).
$$
Recall we have shown  that the full second-order space is spanned by Poincar\'e series.
For example, with (\ref{span1}) and Proposition \ref{ord2gen} we have
\begin{equation}\label{fullspan}
S^2_k = \left\langle \mathcal S^1_{{\text{\rm Par}}}(\ca) \cup \mathcal S^2_{{\text{\rm
Par}}}(\ca) \right\rangle
\end{equation}
in the parabolic case and similarly for the elliptic and hyperbolic series.
Thus, by (\ref{fullspan}) we have
$$
S^2_k = \left\langle \big\{ \Phi_{\text{\rm Par}}(z,m,\ca;L) \ \big| \ m \in \N, \ L \in
\operatorname{Hom}^{[2]}_0(\G,\C) \big\}  \right\rangle.
$$
Looking towards future work, we expect that, for all $n \in \N$,
$$
S^n_k = \left\langle \big\{ \Phi_{\text{\rm Par}}(z,m,\ca;L) \ \big| \ m \in \N, \ L \in
\operatorname{Hom}^{[n]}_0(\G,\C) \big\}  \right\rangle
$$
for a group with a cusp $\ca$, with similar results for the hyperbolic and elliptic cases.
This approach should parallel that of Diamantis and Sim \cite{DS} and also be valid in the
compact case.

\bibliography{refdata2}

\vskip .1in \noindent
\"Ozlem Imamo$\bar{\text g}$lu, Dept. of Math., ETH Zurich, CH-8092, Zurich, Switzerland.
\newline e-mail: ozlem@math.ethz.ch

\vskip .1in
\noindent
Cormac O'Sullivan, Dept. of Math., Bronx Community College, Univ. Ave and W 181 St.,
Bronx, NY 10453. e-mail: cormac12@juno.com

\end{document}